\documentclass[11pt
 ]{elsarticle}
\usepackage{amsmath}
\usepackage{amssymb}
\usepackage{amscd}
\usepackage{amsfonts}
\usepackage{amsthm}
\usepackage{graphicx}
\usepackage{placeins}
\usepackage{color}
\usepackage{float}
\usepackage{caption}
\usepackage[colorlinks, linkcolor=red, citecolor=green,filecolor=magenta,urlcolor=cyan]{hyperref}
\newfont{\Blackboard}{msbm10 scaled 1200}

\newfont{\roma}{cmr10 scaled 1200}

\newtheorem{theorem}{{}\hskip\parindent Theorem}[section]
\newtheorem{lemma}{{}\hskip\parindent Lemma}[section]
\newtheorem{proposition}{{}\hskip\parindent Proposition}[section]
\newtheorem{example}{{}\hskip\parindent Example}[section]

\newtheorem{definition}{{}\hskip\parindent Definition}[section]
\newtheorem{remark}{{}\hskip\parindent Remark}[section]

\def\dfrac{\displaystyle\frac}

\def\be{\begin{equation}}
\def\ee{\end{equation}}
\def\beq{\arraycolsep=1.5pt\begin{eqnarray}}
\def\eeq{\end{eqnarray}}

\journal{...}

\begin{document}
\begin{frontmatter}

\title{ A Palais--Smale type condition for optimal control problems}

\author[Coron]{Jean-Michel Coron}
\address[Coron]{Sorbonne Universit\'e, CNRS, Universit\'e Paris-Cit\'e, Laboratoire Jacques-Louis Lions, LJLL, INRIA \'equipe CAGE, Paris, 75005, France. E-mail
address: jean.m.coron@gmail.com}

\author[Deng]{Li Deng}
\address[Deng]{School of Mathematics,   Southwest Jiaotong University, Chengdu 611756, China. E-mail
address: dengli@swjtu.edu.cn}

\author[Zhang]{Xu Zhang}
\address[Zhang]{School of Mathematics and New Cornerstone Science Laboratory,
Sichuan University, Chengdu, 610064, China. E-mail
address: zhang\_xu@scu.edu.cn.}

\begin{abstract}
We consider an optimal control problem with a general control constraint and a fixed initial state, for which convexity assumptions are absent. To establish the existence of optimal pairs, we introduce a control version of the Palais--Smale-type condition  adapted from critical point theory and prove that optimal pairs exist under this condition (see Theorem~\ref{peo}). Subsequently, we verify that a specific class of optimal control problems satisfies our Palais--Smale-type condition (see Proposition~\ref{scspro}). Furthermore, as an application, we present an example (Example~\ref{ex31}) and show that, to the best of our knowledge, no existing result can be applied to establish the existence of an optimal pair in this case (see Section~\ref{srse31}).   Example~\ref{evpspmp} shows that finitely many Pontryagin maximum principle
candidates do not guarantee the existence of an optimal control. Finally, Example~\ref{edsps} demonstrates that the proposed Palais--Smale-type condition is sufficient, but not necessary, for the existence of optimal controls.
\end{abstract}

\begin{keyword} Optimal control;
Palais--Smale-type condition; Existence of optimal pairs
\MSC[2020]49J15\sep  49J45
\end{keyword}

\end{frontmatter}

\setcounter{equation}{0}

\section{Introduction}
\def\theequation{1.\arabic{equation}}
\hskip\parindent

In optimal control problems, the proof of existence of optimal pairs usually depends on certain convexity conditions. If the problem is not convex, a standard approach is to introduce an associated relaxed problem, prove the existence of optimal pairs for this relaxed problem, and subsequently use them to construct optimal pairs for the original problem.

For example, in optimal control problems governed by ordinary differential equations, \cite{lou9, Neu63, PJ2009} considered cases with free terminal states, whereas \cite{Balder, Ray90} treated problems with terminal state constraints. This list of such references is only partial. Nevertheless, to the best of our knowledge, there exist seemingly simple optimal control problems---such as Example~\ref{ex31}---for which the existence of an optimal pair cannot be established using existing results, as discussed in Section~\ref{srse31}.

At a fundamental level, optimal control problems can be viewed as a special class of
optimization problems. In variational methods and critical point theory, a
central compactness assumption is the Palais--Smale (PS) condition. Roughly
speaking, for a \(C^1\) functional \(\Phi\) on an infinite-dimensional
space \(Z\), the (PS) condition at a level $C\in\mathbb R$ requires that every sequence
\(\{z_k\}\subset Z\) satisfying
\[
    \Phi(z_k)\to C,
    \qquad
    \Phi'(z_k)\to0,
\]
admits a convergent subsequence in \(Z\). This compactness condition prevents minimizing or critical sequences from escaping without converging. It is therefore widely used as a sufficient condition for proving the existence of critical points or minimizers; see, for example,
\cite[Theorem 4.4, p.~82]{MR982267}.

Inspired by this viewpoint, in this paper we propose to adapt the (PS) condition framework to optimal control
problems by introducing an appropriate Palais--Smale-type ((PS)-type) condition that ensures the existence of
optimal pairs. This allows us to prove the existence of optimal
pairs without imposing the usual convexity assumptions on the control
system.

In what follows, we state the optimal control problem under consideration.

Let $T>0$,  let $n$ be a positive integer, and let $U$ be a separable metric space. Denote by $\mathbb R^n$ the set of all real
$n$-dimensional column vectors. Suppose that $f=(f_1,\cdots,f_n)^\top:[0,T]\times\mathbb R^n\times U\to\mathbb R^n$, $f^0: [0,T]\times\mathbb R^n\times U\to\mathbb R$ and $h:\mathbb R^n\to\mathbb R$ are suitable maps.
Given $x_0 \in\mathbb R^n$, we consider the following optimal control problem:
\begin{description}
\item[$(P)$] Minimize
\begin{align*}
J(u(\cdot))=\int_0^T f^0(t,x(t),u(t))d t+h(x(T))
\end{align*}
 over all pairs $(u(\cdot), x(\cdot))$ satisfying the following
\begin{align}\label{cs}
&\left\{\begin{array}{l}
\dot{x}(t)=f(t, x(t), u(t)),\quad u(t)\in U,\quad a.e.\,t\in(0,T),
\\ x(0)=x_0,
\end{array} \right.
\end{align}
\end{description}
For this problem, we set
\begin{equation}\label{uab}
\mathcal U=\{u:[0,T]\to U\mid u(\cdot)\text{ is measurable}\}.
\end{equation}
A pair $(u(\cdot),x(\cdot))$ is said to be  admissible  if it satisfies (\ref{cs}),    $u(\cdot)\in \mathcal U$  and $J(u(\cdot))<+\infty$. We denote by $\mathcal P_{ad}$ the set of all admissible pairs. If an admissible pair $(\bar u(\cdot), \bar x(\cdot))$  satisfies $J(\bar u(\cdot))=\inf_{u(\cdot)\in\mathcal U}J(u(\cdot))$,  then the pair \((\bar{u}(\cdot), \bar{x}(\cdot))\) and the control \(\bar{u}(\cdot)\) are referred to as an optimal pair and an optimal control, respectively.

 The application of critical point theory to optimal control problems has been explored in the earlier literature. For example, Vakhrameev \cite{Vakhrameev94} investigated optimal control problems in which the terminal state is constrained to lie on a submanifold with corners. Under suitable conditions, such problems can be transformed into optimization problems defined on infinite-dimensional manifolds with corners, thereby allowing established critical point theory to be applied. In particular, the existence of optimal solutions is guaranteed when the corresponding infinite-dimensional (PS) condition holds. However, this approach has limitations: transforming the original optimal control problem into an infinite-dimensional optimization problem requires several additional assumptions. Among them, three fundamental conditions must be met:
\begin{itemize}
\item[1)] there is no constraint on the controls;
\item[2)] the terminal state constraint must define a submanifold with corners;
\item[3)] this submanifold must be contained in the reachable set of the control system.
\end{itemize}
Clearly, our problem $(P)$ generally does not satisfy these requirements.

In this work, we adapt methodologies from critical point theory, specifically by directly applying Ekeland's Variational Principle to the control problem (P). This necessitates introducing an appropriate control-theoretic version of the (PS)-type condition (see Definition~\ref{pstc}), which ensures the existence of optimal pairs.

This paper is organized as follows. Section~\ref{mre} contains the main results in this work. Section~\ref{threxs} is devoted to examples. The proofs of our main results are given in Section~\ref{pmr}.

\section{Main results }\label{mre}
\def\theequation{2.\arabic{equation}}
\hskip\parindent

\subsection{The (PS)-type condition (of control version)}
In the rest of this paper, we impose the following assumption.
\begin{description}
\item[(A)]   The maps \( f: [0, T] \times \mathbb{R}^n \times U \to \mathbb{R}^n \) and \( f^0: [0, T] \times \mathbb{R}^n \times U \to \mathbb{R} \) are measurable in \( t \in [0, T] \), and of class \( C^1 \) in \( x \in \mathbb{R}^n \). The map \( h: \mathbb{R}^n \to \mathbb{R} \) is also \( C^1 \). For a.e. \( t \in [0, T] \), the maps \( f(t, \cdot, \cdot) \), \( \nabla_x f(t, \cdot, \cdot) \), \( f^0(t, \cdot, \cdot) \), and \( \nabla_x f^0(t, \cdot, \cdot) \) are continuous on $\mathbb R^n\times U$, where \( \nabla_x f(t, x, u) \) and \( \nabla_x f^0(t, x, u) \) denote the partial derivatives of \( f \) and \( f^0 \) with respect to \( x \), respectively.

There exist  positive constants \( K \),  $K_1$ and $K_2$ such that
\begin{align}\label{bnf1}
&|\nabla_x f(t, x, u)|\leq K,\quad \, |f(t, 0, u)|  \leq K_1,
\\& \label{bfz7} |f^0(t, 0, u)|\leq K_2,
\end{align}
hold for almost every \( t \in [0, T] \) and all \( (x, u) \in \mathbb{R}^n \times U \).

Moreover, for any \( r > 0 \), there exists a positive constant \( C^0_r  \) such that
\begin{align}\label{bnfz1}
|\nabla_x f^0(t, x, u)| \leq C^0_r, \quad a.e. \,t\in [0, T], \;\forall\; (x, u) \in \mathbb{R}^n \times U \text{ with } |x| \leq r.
\end{align}
For any \( \epsilon > 0 \), there exists \( \delta > 0 \), depending on \( \epsilon \) and \( r \), such that for all \( y \in \mathbb{R}^n \) with \( |y| \leq \delta \),
\begin{equation}\label{bnfz2}
|\nabla_x f^0(t, x + y, u) - \nabla_x f^0(t, x, u)| < \epsilon, \; \forall\, (x, u) \in  \mathbb{R}^n \times U \text{ with } |x| \leq r,
\end{equation}
holds for a.e. $t\in [0, T]$.
\end{description}

Denote by $(\mathbb R^n)^*$ the set of all real
$n$-dimensional vectors. The Hamiltonian function $H(\cdot,\cdot,\cdot,\cdot): [0, T]\times \mathbb R^n\times(\mathbb R^n)^*\times U\to\mathbb R$ (for the problem $(P)$) is given by
\begin{align}\label{hf1}
H(t,x,p,u)=p f(t,x,u)-f^0(t,x,u),\quad \forall\,(t,x,p,u)\in [0, T]\times \mathbb R^n\times(\mathbb R^n)^*\times U.
\end{align}
The set \(\mathcal U\) (given by \eqref{uab}) becomes a metric space when equipped with the Ekeland metric \(d\), defined by
\begin{equation}\label{ekem}
d(u_1(\cdot),u_2(\cdot))=\bigl|\{t\in[0,T]\mid u_1(t)\neq u_2(t)\}\bigr|,
\end{equation}
where \(|A|\) denotes the Lebesgue measure of a Lebesgue measurable subset \(A\subset\mathbb R\).

We now propose the following control version of the (PS)-type condition.
\begin{definition}\label{pstc}
The problem $(P)$ is said to satisfy the (PS)-type condition at a level $C\in\mathbb R$ if the following holds: for any sequences $\{u_k(\cdot)\}\subset \mathcal U$ and $\{c_k\}\subset[0,+\infty)$ satisfying
\begin{equation}\label{ps3}
\left\{\begin{array}{l}
\int_0^T\left( H(t, x_k(t), p_k(t),u_k(t))-H(t,x_k(t), p_k(t),u(t)) \right) dt\geq -c_k,
\\
\phantom{\int_0^T\Big( H(t, x_k(t), p_k(t),u_k(t))-\Big)}\forall\; (u(\cdot),k)\in \mathcal U\times \mathbb N^+,
\\
\lim_{k\to+\infty}c_k=0,\quad \lim_{k\to+\infty} J(u_k(\cdot))=C,
\end{array}
\right.
\end{equation}
where $x_k(\cdot)$ is the solution to (\ref{cs}) corresponding to the control $u_k(\cdot)$,
 and $p_k(\cdot)$ is the solution to
\begin{align}\label{dek}
\left\{ \begin{array}{l} \dot p_k(t)=-p_k(t)\nabla_x f(t,x_k(t),u_k(t)) +\nabla_x f^0(t,x_k(t),u_k(t)),\quad a.e.\,t\in(0,T),
\\ p_k(T)=-\nabla h(x_k(T)),
\end{array}
\right.
\end{align}
 there exists a subsequence of $\{u_k(\cdot)\}$  that converges in the metric space $(\mathcal U,d)$.
\end{definition}

\subsection{Main results}

\begin{theorem}\label{peo}
Assume that (A) holds, \(\mathcal U\neq\emptyset\),
$
    C:=\inf_{u(\cdot)\in\mathcal U}J(u(\cdot))
    \in(-\infty,+\infty)$,
and that the problem \((P)\) satisfies the (PS)-type condition at \(C\).
Then the problem \((P)\) admits at least one optimal pair.
\end{theorem}

\begin{remark}\label{rem:PS_sufficient_not_necessary}
The \((PS)\)-type condition in Theorem \ref{peo} is a
sufficient compactness condition for the existence of optimal controls, but it is not a necessary one.  This is analogous to the role of the
classical \((PS)\) condition in critical point theory: it is a compactness hypothesis used to obtain existence results, but it is not itself necessary for the existence of critical points. See Example \ref{edsps} for the corresponding phenomenon in the present optimal control setting.
\end{remark}

\medskip

Next, we show that a class of optimal control problems satisfies the
\((PS)\)-type condition.  For this purpose, we need the
following result.

\begin{lemma}\label{bxec}
Assume that  (A) holds. For any \(u(\cdot)\in\mathcal U\),
denote by \(x_u(\cdot)\) the solution of \eqref{cs} corresponding to
\(u(\cdot)\), and by \(p_u(\cdot)\) the solution of
\begin{align}\label{dek_u}
\left\{
\begin{array}{l}
\dot p_u(t)
=
-p_u(t)\nabla_x f(t,x_u(t),u(t))
+
\nabla_x f^0(t,x_u(t),u(t)),
\quad a.e.\,t\in(0,T),
\\[1mm]
p_u(T)=-\nabla h(x_u(T)).
\end{array}
\right.
\end{align}
Then
\begin{align}\label{bxu11}
|x_u(t)|\le R(t),\qquad |p_u(t)|\le P(t),
    \quad t\in[0,T],
\end{align}
where
\[
R(t)
=
|x_0|e^{Kt}
+
K_1\frac{e^{Kt}-1}{K},
\]
and
\[
P(t)
=
\sup_{|x|\leq R(T)}|\nabla h(x)|\, e^{K(T-t)}
+
C_{R(T)}^0\frac{e^{K(T-t)}-1}{K}.
\]
Moreover, the family
\(
    \{(x_u(\cdot),p_u(\cdot))\mid u(\cdot)\in\mathcal U\}
\)
is relatively compact in
\(
C([0,T];\mathbb R^n)\times C([0,T];(\mathbb R^n)^*).
\)
\end{lemma}

\begin{proposition}
\label{scspro}
Let \(T>0\), \(x_0\in\mathbb R^n\), \(m\in\mathbb N^+\setminus\{1\}\), $C\in\mathbb R$,  and let
$
    U=\{v_1,\ldots,v_m\}
$
be a finite control set.  Assume that (A) holds. Assume further that the following condition is satisfied:

\begin{enumerate}
\item[(H)] For any sequences \(\{u_k(\cdot)\}\subset\mathcal U\) and
\(\{c_k\}\subset[0,\infty)\) satisfying \eqref{ps3}, let \(x_k(\cdot)\) be the solution of \eqref{cs} corresponding to the control \(u_k(\cdot)\), and let \(p_k(\cdot)\) be the solution of \eqref{dek}. By Lemma \ref{bxec}, there exist a subsequence of
\(
\{(x_k(\cdot),p_k(\cdot))\}_{k\ge1}\),
still denoted by itself, and
\(
(x_*(\cdot),p_*(\cdot))
\in
C([0,T];\mathbb R^n)\times C([0,T];(\mathbb R^n)^*)
\)
such that
\(
(x_k(\cdot),p_k(\cdot))
\to
(x_*(\cdot),p_*(\cdot))
\) uniformly on \([0,T]\).
For every such limiting pair \((x_*(\cdot),p_*(\cdot))\), assume that, for a.e. \(t\in[0,T]\), the function
\[
v\mapsto H(t,x_*(t),p_*(t),v),\qquad v\in U,
\]
has a unique maximizer over \(U\). Equivalently, if the numbers 
\[
    H(t,x_*(t),p_*(t),v_1),\ldots,
    H(t,x_*(t),p_*(t),v_m)
\]
are arranged in a non-increasing order
\(
    \lambda_1(t)\geq\lambda_2(t)\geq\cdots\geq\lambda_m(t)
\),
and if
\begin{align}
    \label{deta8}
    \Delta_*(t):=\lambda_1(t)-\lambda_2(t),
\end{align}
then
\begin{align}\label{D80}
    \left|\{t\in[0,T]\mid \Delta_*(t)=0\}\right|=0.
\end{align}
\end{enumerate}

Then the problem \((P)\) satisfies the (PS)-type condition at \(C\).
\end{proposition}

\setcounter{equation}{0}
\section{Examples}\label{threxs}
\def\theequation{3.\arabic{equation}}
\hskip\parindent
\subsection{Example satisfying the (PS)-type condition }

\begin{example}\label{ex31} Let \(T=1\), \(U=\{1,2,3,4\}\), and
\begin{align}\label{mcs1}
\mathcal U
    =
    \{u:[0,1]\to U\mid u(\cdot)\ {\rm is\ measurable}\}.
\end{align}
 We consider the control system \begin{equation}\label{eq:state_dependent_example} \left\{ \begin{array}{l} \dot x_1(t)=x_1(t)\sin u(t)+u(t),\\[1mm] \dot x_2(t)=x_2(t)\cos u(t)+u(t)^2,\\[1mm] x(0)=0, \end{array} \right. \qquad u(t)\in U,\quad a. e. \,t\in[0,1]. \end{equation} The cost functional is \begin{equation}\label{eq:J_state_dependent_example} J(u(\cdot)) = \int_0^1 Mt\,u(t)\,dt + \sin x_1^u(1)+\cos x_2^u(1), \end{equation} where \(x^u(\cdot)\) is the solution of \eqref{eq:state_dependent_example} corresponding to \(u(\cdot)\), and \begin{align}\label{lbm} M>(2e-1) \bigl[ 4(\sin 1-\sin 4) + 16(\cos 1-\cos 4) \bigr] + 18. \end{align} Then the above optimal control problem admits at least one optimal control. \end{example}

\begin{proof} We verify the assumptions of Proposition \ref{scspro} and then apply Theorem \ref{peo}.

 {\it Step 1.}\quad Set \begin{align*} f(t,x,u) = \begin{pmatrix} x_1\sin u+u\\[1mm] x_2\cos u+u^2 \end{pmatrix}, \quad f^0(t,x,u)=Mtu, \quad h(x)=\sin x_1+\cos x_2. \end{align*}
 The functions \(f\), \(f^0\), and \(h\) are smooth with respect to \(x\), and
\(U=\{1,2,3,4\}\) is finite. Hence the measurability and continuity assumptions
in (A) are immediate. In addition,
\(\nabla_x f(t,x,u)=\mathrm{diag}(\sin u,\cos u)\) is bounded by \(K=1\) on
\([0,1]\times \mathbb R^2\times U\). Since
\(f(t,0,u)=(u,u^2)^\top\) and \(f^0(t,0,u)=Mtu\) are bounded on
\([0,1]\times U\), and \(\nabla_x f^0(t,x,u)\equiv0\), condition (A) holds.

{\it Step 2.}\quad For any \(u(\cdot)\in\mathcal U\), we claim \begin{align}\label{xu1} & |x_{1}^u(t)| \leq 4(e^t-1),\quad |x_{2}^u(t)| \leq 16(e^t-1), \qquad t\in[0,1], \\\label{dxu1} & |\dot x_{1}^u(t)| \leq 4e^t,\quad |\dot x_{2}^u(t)| \leq 16e^t, \qquad \text{for a.e. }t\in[0,1], \\\label{pu1} & |p_{1}^u(t)| \leq e^{1-t},\quad |p_{2}^u(t)| \leq e^{1-t},\quad t\in[0,1], \\ \label{dpu1} & |\dot p_{1}^u(t)| \leq e^{1-t},\quad |\dot p_{2}^u(t)| \leq e^{1-t},\quad \text{for a.e. }t\in[0,1], \end{align} where \[ x^u(\cdot)=(x_1^u(\cdot),x_2^u(\cdot)) \] is the solution to \eqref{eq:state_dependent_example} corresponding to the control \(u(\cdot)\), and \[ p^u(\cdot)=(p_1^u(\cdot),p_2^u(\cdot)) \] is the solution to the adjoint equation \begin{align}\label{pu6} \left\{ \begin{array}{l} \dot p_{1}^u(t)=-p_{1}^u(t)\sin u(t), \\[1mm] \dot p_{2}^u(t)=-p_{2}^u(t)\cos u(t), \\[1mm] p^u(1)=(-\cos x_{1}^u(1),\sin x_{2}^u(1)). \end{array} \right. \end{align} Indeed, since \(u(t)\in\{1,2,3,4\}\) for a.e. \(t\in[0,1]\), we have \[ |\sin u(t)|\leq 1,\qquad |\cos u(t)|\leq 1, \qquad |u(t)|\leq 4,\qquad |u(t)|^2\leq 16, \quad \text{for a.e. }t\in[0,1]. \] Hence \begin{align}\label{dx1dx2} |\dot x_{1}^u(t)| \leq |x_{1}^u(t)|+4, \qquad |\dot x_{2}^u(t)| \leq |x_{2}^u(t)|+16, \quad \text{for a.e. }t\in[0,1]. \end{align} Equivalently, \[ |x_{1}^u(t)| \leq \int_0^t |x_{1}^u(s)|\,ds+4t, \quad |x_{2}^u(t)| \leq \int_0^t |x_{2}^u(s)|\,ds+16t, \quad t\in[0,1]. \] By Gronwall's inequality, we derive \eqref{xu1}. Consequently, \eqref{dxu1} follows from \eqref{dx1dx2} and \eqref{xu1}.

We now prove \eqref{pu1}. We make the time-reversal change of variables \[ \tau=1-t, \qquad \widetilde p_{1}^u(\tau):=p_{1}^u(1-\tau), \qquad \tau\in[0,1]. \] Then \[ \begin{aligned} \dot{\widetilde p}_{1}^u(\tau) &= -\bigl[-p_{1}^u(1-\tau)\sin u(1-\tau)\bigr] \\ &= \widetilde p_{1}^u(\tau)\sin u(1-\tau), \qquad \text{for a.e. }\tau\in(0,1). \end{aligned} \] Moreover, \[ \widetilde p_{1}^u(0) = p_{1}^u(1) = -\cos x_{1}^u(1), \] and therefore $ |\widetilde p_{1}^u(0)|\leq 1$.  Integrating the equation for \(\widetilde p_{1}^u\) from \(0\) to \(\tau\), we get \[ \widetilde p_{1}^u(\tau) = \widetilde p_{1}^u(0) + \int_0^\tau \widetilde p_{1}^u(s)\sin u(1-s)\,ds. \] Thus, \[ |\widetilde p_{1}^u(\tau)| \leq 1+ \int_0^\tau |\widetilde p_{1}^u(s)|\,ds. \] By Gronwall's inequality, \[ |\widetilde p_{1}^u(\tau)| \leq e^\tau, \qquad \tau\in[0,1]. \] Returning to the original time variable \(t=1-\tau\), we obtain \[ |p_{1}^u(t)| = |\widetilde p_{1}^u(1-t)| \leq e^{1-t}, \qquad t\in[0,1]. \] Similarly, we estimate \(p_2^u(t)\) in the same way and obtain \[ |p_{2}^u(t)| \leq e^{1-t}, \qquad t\in[0,1]. \] Hence \eqref{pu1} holds.

Finally, \[ |\dot p_{1}^u(t)| = |p_{1}^u(t)\sin u(t)|, \quad |\dot p_{2}^u(t)| = |p_{2}^u(t)\cos u(t)|, \quad \text{for a.e. }t\in[0,1], \] which, together with \eqref{pu1}, implies \eqref{dpu1}.

{\it Step 3.}\quad Denote by $ C:=\inf_{u(\cdot)\in\mathcal U}J(u(\cdot))$. Since \(u(t)\in\{1,2,3,4\}\) for almost every $t\in[0,1]$,
we have \[ \frac{M}{2}-2 \leq J(u(\cdot)) \leq 2M+2, \qquad \forall\,u(\cdot)\in\mathcal U. \] Thus \(C\in(-\infty,+\infty)\). Let \(\{u_k(\cdot)\}\subset\mathcal U\)
 and \(\{c_k\}\subset [0,+\infty)\) satisfy \eqref{ps3}. Let \(x_k(\cdot)\) be
 the solution to \eqref{eq:state_dependent_example} corresponding to \(u_k(\cdot)\), and let
  \(p_k(\cdot)\) be the solution to \eqref{pu6} corresponding to \(u_k(\cdot)\). By
  Lemma \ref{bxec}, after passing to a subsequence if necessary, we may suppose that \begin{align*} (x_k(\cdot),p_k(\cdot)) \to (x_*(\cdot),p_*(\cdot)) \quad \text{uniformly on }[0,1]. \end{align*} We claim that \begin{align}\label{x8e} &|x_{*,1}(t)|\leq 4(e^t-1),\quad |x_{*,2}(t)| \leq 16(e^t-1), \quad \forall\, t\in[0,1], \\\label{p8e} & |p_{*,1}(t)| \leq e^{1-t},\quad |p_{*,2}(t)| \leq e^{1-t},\quad t\in[0,1], \\\label{eq:component_dot_xstar_bounds} &|\dot x_{*,1}(t)| \leq 4e^t,\quad |\dot x_{*,2}(t)|\leq 16e^t, \qquad \text{for a.e. }t\in[0,1], \\\label{eq:component_dot_pstar_bounds} &|\dot p_{*,1}(t)| \leq e^{1-t}, \quad |\dot p_{*,2}(t)| \leq e^{1-t}, \qquad \text{for a.e. }t\in[0,1]. \end{align} In fact, passing to the limit, we obtain \eqref{x8e} and \eqref{p8e} from \eqref{xu1} and \eqref{pu1}. Then, by \eqref{dxu1}, for every \(0\leq s<t\leq 1\), we have \begin{align*} |x_{k,1}(t)-x_{k,1}(s)| \leq \int_s^t 4e^\tau\,d\tau, \quad |x_{k,2}(t)-x_{k,2}(s)| \leq \int_s^t 16e^\tau\,d\tau. \end{align*} Passing to the limit, we obtain \begin{align*} |x_{*,1}(t)-x_{*,1}(s)| \leq \int_s^t 4e^\tau\,d\tau, \quad |x_{*,2}(t)-x_{*,2}(s)| \leq \int_s^t 16e^\tau\,d\tau. \end{align*} Therefore \(x_{*,1}(\cdot)\) and \(x_{*,2}(\cdot)\) are absolutely continuous and satisfy \eqref{eq:component_dot_xstar_bounds}. Similarly, following the same argument and using \eqref{dpu1}, we obtain \[ |p_{*,1}(t)-p_{*,1}(s)| \leq \int_s^t e^{1-\tau}\,d\tau, \quad |p_{*,2}(t)-p_{*,2}(s)| \leq \int_s^t e^{1-\tau}\,d\tau. \] Hence \(p_{*,1}(\cdot)\) and \(p_{*,2}(\cdot)\) are absolutely continuous and satisfy \eqref{eq:component_dot_pstar_bounds}.

{\it Step 4.}\quad We show that the condition (H) in Proposition
\ref{scspro} is satisfied. Consequently, it follows from Theorem \ref{peo} that there exists at least one optimal pair.

To this end, let $\{u_k(\cdot)\}\subset \mathcal U$ and $\{c_k\}\subset[0,+\infty)$ be given in Step 3.

For this problem, the Hamiltonian function is
\begin{align}\label{eq:Hamiltonian_state_dependent_example}
    H(t,x,p,u)
    &=
    p f(t,x,u)-f^0(t,x,u)
    \notag\\
    &=
    p_1(x_1\sin u+u)
    +
    p_2(x_2\cos u+u^2)
    -
    Mtu,
\end{align}
for $(t,x,p,u)\in[0,1]\times\mathbb R^2\times(\mathbb R^2)^*\times U$.
For \(v,w\in U\) with \(v\neq w\), set
\[\begin{aligned}
    G_{vw}(t)
    &:=
    H(t,x_*(t),p_*(t),v)
    -
    H(t,x_*(t),p_*(t),w),
    \\&=A_{vw}(t)-Mt(v-w),\quad \forall\,t\in[0,1],
    \end{aligned}
\]
where
\begin{align*}
    A_{vw}(t)
    &:=
    p_{*,1}(t)x_{*,1}(t)(\sin v-\sin w)
    +
    p_{*,1}(t)(v-w)
    \\
    &\quad
    +
    p_{*,2}(t)x_{*,2}(t)(\cos v-\cos w)
    +
    p_{*,2}(t)(v^2-w^2).
\end{align*}
Then, for almost every $t\in[0,1]$,
\[
    \dot G_{vw}(t)
    =
    \dot A_{vw}(t)-M(v-w).
\]

We now give an explicit estimate for \(|\dot A_{vw}(t)|\).
From \eqref{x8e}, \eqref{p8e},  \eqref{eq:component_dot_xstar_bounds} and
\eqref{eq:component_dot_pstar_bounds}, we have
\[
\begin{aligned}
\left|
\frac{d}{dt}\bigl(p_{*,1}(t)x_{*,1}(t)\bigr)
\right|
&\leq
|\dot p_{*,1}(t)|\,|x_{*,1}(t)|
+
|p_{*,1}(t)|\,|\dot x_{*,1}(t)|
\\
&\leq
e^{1-t}\cdot 4(e^t-1)
+
e^{1-t}\cdot 4e^t
\\
&=
4(2e-e^{1-t}),\qquad a.e.\,t\in[0,1],
\end{aligned}
\]
and similarly
\[
\left|
\frac{d}{dt}\bigl(p_{*,2}(t)x_{*,2}(t)\bigr)
\right|
\leq
16(2e-e^{1-t}),\quad a.e.\,t\in[0,1].
\]
Therefore,  we
set
\[
B_{vw}
:=
4|\sin v-\sin w|
+
16|\cos v-\cos w|,\quad  D_{vw}
:=
|v-w|+|v^2-w^2|,
\]
and obtain
\[
\begin{aligned}
|\dot A_{vw}(t)|
&\leq
4(2e-e^{1-t})|\sin v-\sin w|
+
e^{1-t}|v-w|
\\
&\quad
+
16(2e-e^{1-t})|\cos v-\cos w|
+
e^{1-t}|v^2-w^2|
\\&=e^{1-t}(D_{vw}-B_{vw})+2eB_{vw},\quad a.e.\,t\in[0,1].
\end{aligned}
\]

Since \(e^{1-t}\in[1,e]\), the preceding inequality yields
\[
|\dot A_{vw}(t)|
\leq
\max\left\{
(2e-1)B_{vw}+D_{vw},
\;
eB_{vw}+eD_{vw}
\right\},\quad a.e.\,t\in[0,1].
\]
Thus
\[
\max_{\substack{v,w\in\{1,2,3,4\}\\ v\neq w}}|\dot A_{vw}(t)|
\leq
\max_{\substack{v,w\in\{1,2,3,4\}\\ v\neq w}}
\max\left\{
(2e-1)B_{vw}+D_{vw},
\;
eB_{vw}+eD_{vw}
\right\},
\]
for almost every $t\in[0,1]$.
A direct computation over the finite set \(U=\{1,2,3,4\}\) shows that
\begin{align}\label{avwb}\begin{array}{ll}
\displaystyle\max_{\substack{v,w\in\{1,2,3,4\}\\ v\neq w}}|\dot A_{vw}(t)|
&\leq (2e-1)B_{14}+D_{14
}
\\[2mm]
&=(2e-1)
\bigl[
4(\sin 1-\sin 4)
+
16(\cos 1-\cos 4)
\bigr]
+
18,
\end{array}
\end{align}
for almost every $t\in[0,1]$.

Next,
if \(v>w\), then \(v-w\geq1\), and hence, by
\eqref{avwb} and \eqref{lbm},
\[
    \dot G_{vw}(t)
    \leq
   |\dot A_{vw}(t)| -M<0
    \quad
    \text{for a.e. }t\in[0,1].
\]
If \(v<w\), then \(w-v\geq1\), and therefore
\[
    \dot G_{vw}(t)
    \geq
    -|\dot A_{vw}(t)|+M>0
    \quad
    \text{for a.e. }t\in[0,1].
\]
Hence \(G_{vw}\) is strictly monotone on \([0,1]\). In particular,
\[
    \big|\{t\in[0,1]\mid G_{vw}(t)=0\}\big|=0.
\]
Since \(U\) is finite, we obtain
\[\begin{aligned}
    &\left|
    \left\{
        t\in[0,1]\mid
        H(t,x_*(t),p_*(t),v)
        =
        H(t,x_*(t),p_*(t),w)
        \text{ for some }v\neq w
    \right\}
    \right|
    \\&=\left| \bigcup_{\substack{v,w\in U,\\  v\neq w}}\{t\mid G_{vw}(t)=0\}
    \right|
    \\&=0.
\end{aligned}
\]
Equivalently, for a.e. \(t\in[0,1]\), the function
\[
    v\mapsto H(t,x_*(t),p_*(t),v),
    \qquad v\in U,
\]
has a unique maximizer. Thus the condition (H) in Proposition
\ref{scspro} is satisfied.
This completes the proof.
\end{proof}

\subsubsection{Some remarks on Example \ref{ex31}}\label{srse31}

\begin{remark}
The works \cite{Balder,Ray90} study the existence of optimal pairs for
optimal control problems with non-convex structures.  In particular,
\cite[Corollary~5.2]{Balder} applies to control systems of the form
\begin{align}\label{ybyb2}
 \dot y(t)=B(t)y(t)+b(t,u(t)),
    \qquad {\rm a.e.}\ t\in[t_0,t_1],
\end{align}
where \(t_0<t_1\), the matrix \(B(t)\) is independent of the control $u$, and
\(b(t,u)\) is independent of the state variable $y$.  Therefore, the coefficient
of the state variable is allowed to depend on time, but not on the control.

However, the control system \eqref{eq:state_dependent_example} in
Example~\ref{ex31} does not have this structure.
Indeed, it can be rewritten as
\[
    \dot x(t)=A(u(t))x(t)+\beta(u(t)),
\]
where
\[
    A(v)=
    \begin{pmatrix}
        \sin v & 0\\
        0 & \cos v
    \end{pmatrix},
    \qquad
    \beta(v)=
    \begin{pmatrix}
        v\\
        v^2
    \end{pmatrix},
    \qquad v\in U.
\]
Thus, the coefficient matrix of the state variable depends explicitly on the
control value \(v\).  Hence the system in
Example~\ref{ex31} is not covered by
\cite[Corollary~5.2]{Balder}.  Consequently, that corollary cannot be used to
deduce the existence of an optimal control for this example.


Theorem~4.2 in \cite{Ray90} also deals with systems of the form
\eqref{ybyb2}, together with the cost functional
\[
J(u(\cdot))=\int_{t_0}^{t_1}\Big(f^0(t,u(t))+g(t,y(t))\Big)\,dt
+h(y(t_1)).
\]
Here \(f^0:[t_0,t_1]\times\mathbb R^m\to\mathbb R\), with
\(m\in\mathbb N^+\), is independent of the state variable;
\(g:[t_0,t_1]\times\mathbb R^n\to\mathbb R\), with \(n\in\mathbb N^+\),
is independent of the control variable; and
\(h:\mathbb R^n\to\mathbb R\) is lower semicontinuous and bounded below.

To guarantee the existence of an optimal pair, \cite[Theorem~4.2]{Ray90}
imposes the following additional assumptions on \(f^0\) and \(g\):
\begin{description}
    \item[(i)] For almost every \(t\in[t_0,t_1]\), the function
    \(g(t,\cdot)\) is concave.

    \item[(ii)] Let \(q\in[1,+\infty)\). Define
    \(\Phi_q:\mathbb R\to[0,+\infty)\) as follows. If \(q>1\), then
    \(\Phi_q(\xi)=|\xi|^q\) for \(\xi\in\mathbb R\). If \(q=1\), then
    \(\Phi_1\) is assumed to be even and convex, and satisfies
    $
    \lim_{\zeta\to+\infty}\frac{\Phi_1(\zeta)}{\zeta}=+\infty$.
    Moreover, there exist \(a_1\in L^1(t_0,t_1)\), \(c_1>0\), and a
    function \(\Phi:[t_0,t_1]\times\mathbb R\to[0,+\infty)\) satisfying
    \[
    \lim_{\zeta\to+\infty}\Phi(t,\zeta)=+\infty,
    \quad \forall\,t\in[t_0,t_1],
    \]
    such that
    \[
    \Phi(t,|u|)+c_1\Phi_q(|b(t,u)|)+a_1(t)
    \leq f^0(t,u),
    \quad \forall\,(t,u)\in[t_0,t_1]\times\mathbb R^m.
    \]
\end{description}

The assumptions above cannot be satisfied simultaneously in Example~\ref{ex31}. Therefore,
\cite[Theorem~4.2]{Ray90} is not applicable to our setting.
\end{remark}

The following proposition follows directly from  \cite[Theorem 4.2, p.58]{Berkovitz} and Lemma  \ref{bxec}.
\begin{proposition} Assume that  (A) holds. Denote by $\overline{co}\,A$ the closed convex hull of a set $A$. Assume the weak Cesari property holds at any point $(t,x)\in[0,T]\times\mathbb R^n$:
\begin{align*}
\bigcap_{\delta>0}\overline{co}\,\mathcal E(t,\mathcal O_\delta(x))=\mathcal E(t,x),
\end{align*}
where $\mathcal O_\delta(x):=\{y\in\mathbb R^n\mid\,|y-x|<\delta\}$,
$$\mathcal E(t,y):=\{(z^0,z)\in\mathbb R\times \mathbb R^n\mid z=f(t,y,u),\,z^0\geq f^0(t,y,u),\,u\in U\},$$
for $(t,y)\in[0,T]\times\mathbb R^n$, and $\mathcal E(t,\mathcal O_\delta(x)):=\cup_{y\in\mathcal O_\delta(x)}\mathcal E(t,y)$.
Then, problem (P) has at least one optimal pair.

\end{proposition}

\begin{remark}\label{rem:convexified_Cesari_failure}
Example \ref{ex31} cannot be obtained from the existence result based on the
weak Cesari property.  Indeed, the weak Cesari property fails not only for
the original finite control set \(U=\{1,2,3,4\}\), but also for its convex
hull \([1,4]\).

Let
\[
    \widetilde U=\{1,2,3,4\}
    \quad\text{or}\quad
    \widetilde U=[1,4].
\]
For \(t_0\in(0,1]\) and \(x=0\), set
\[
\mathcal E_{\widetilde U}(t_0,0)
=
\left\{
    (z^0,z_1,z_2)\in\mathbb R\times\mathbb R^2
    \mid
    z_1=u,\ z_2=u^2,\ z^0\geq Mt_0u,\ u\in\widetilde U
\right\}.
\]
The two points
\begin{align}\label{p1p3}
    P_1=(Mt_0,1,1),
    \qquad
    P_3=(3Mt_0,3,9)
\end{align}
belong to \(\mathcal E_{\widetilde U}(t_0,0)\).  However,
\[
    \frac{P_1+P_3}{2}=(2Mt_0,2,5).
\]
If this midpoint belonged to \(\mathcal E_{\widetilde U}(t_0,0)\), then
\(z_1=u\) would imply \(u=2\).  But then \(z_2=u^2=4\), which contradicts
\(z_2=5\).  Hence
\[
    \frac{P_1+P_3}{2}
    \notin
    \mathcal E_{\widetilde U}(t_0,0).
\]
Therefore \(\mathcal E_{\widetilde U}(t_0,0)\) is not convex.  Since
\[
    \bigcap_{\delta>0}
    \overline{\operatorname{co}}\,
    \mathcal E_{\widetilde U}(t_0,\mathcal O_\delta(0))
\]
is convex, the weak Cesari property cannot hold at \((t_0,0)\).

Consequently, the existence of an optimal control in Example \ref{ex31} is
not a consequence of the weak Cesari property, neither for the original
finite control set nor for its convexification.
\end{remark}

Next, we recall a result from \cite{PJ2009}. The following proposition, which combines \cite[Theorem 1.3]{PJ2009} with a conclusion from the proof of \cite[Proposition 3.1]{PJ2009}, investigates a special class of optimal control problems and provides sufficient conditions for the existence of optimal pairs.
\begin{proposition}\label{pj13}Let $n,s,m\in\mathbb{N}^+$, $T>0$,  and $U\subset\mathbb{R}^m$ be a compact set. Given maps $c_i: \mathbb{R}^n \to \mathbb{R}$, $\phi_i: \mathbb{R}^m \to \mathbb{R}$, $Q_i: \mathbb{R}^n \to \mathbb{R}^n$ (with $i=1,\cdots,s$) such that $c_i$ are continuous, $\phi_i$ are of class $C^1$ and $Q_i$ are Lipschitz, consider the following optimal control problem:

\[
(P_2) \quad
\begin{cases}
\text{Minimize } & \displaystyle \int_0^T \left[ \sum_{i=1}^s c_i(x(t)) \phi_i(u(t)) \right] dt \\
\text{subject to} & \dot{x}(t) = \sum_{i=1}^s Q_i(x(t)) \phi_i(u(t))\quad \text{a.e.  in } (0,T), \\
& x(0) = x_0 \in \mathbb{R}^n,\quad u \in L^\infty(0,T;\mathbb R^m), \\
 & \quad u(t) \in U\;a.e.\,t\in(0,T).
\end{cases}
\]

Set $c=(c_1,\cdots,c_s)$, $\phi=(\phi_1,\cdots,\phi_s)^\top$, and $Q=(Q_1, \cdots, Q_s)$.  Assume for each $x\in\mathbb{R}^n$, and $\xi\in Q(x)\overline{\mathrm{co}}\,\phi(U)$, the minimum
\begin{align}\label{th13}
    \min_{\substack{y\in \overline{\mathrm{co}}\,\phi(U) \\ \xi=Q(x)y}} c(x) y
\end{align}
is attained only in $\phi(U)$. Then, the problem $(P_2)$ admits an optimal pair.
\end{proposition}

\begin{remark}
We next explain why Example~\ref{ex31} cannot be covered by
Proposition~\ref{pj13}, even after the usual state augmentation.

Introduce the extended state
\[
X(t)=(X_0(t),X_1(t),X_2(t))
:=
(t,x_1(t),x_2(t)),\quad t\in [0,1],
\]
set
\[
\phi(v)=
\begin{pmatrix}
1&
\sin v&
v&
\cos v&
v^2
\end{pmatrix}^\top,
\qquad v\in U,
\]
and define
\[
Q(X)
=
\begin{pmatrix}
1 & 0 & 0 & 0 & 0\\
0 & X_1 & 1 & 0 & 0\\
0 & 0 & 0 & X_2 & 1
\end{pmatrix},\quad\forall\,X=(X_0,X_1,X_2)^\top\in\mathbb R^3.
\]
Then the state equation (\ref{eq:state_dependent_example}) becomes
\[
\left\{
\begin{array}{l}
\dot X(t)=Q(X(t))\phi(u(t)),\quad a.e.\,t\in(0,1),
\\ X(0)=(0,0,0).
\end{array}
\right.
\]

Hence the extended dynamics have the form required in Proposition~\ref{pj13}.

It remains to rewrite the terminal cost as an integral term. Along every admissible trajectory,
\[
\begin{aligned}
\frac{d}{dt}\bigl(\sin X_1(t)+\cos X_2(t)\bigr)
&=
\cos X_1(t)\dot X_1(t)
-
\sin X_2(t)\dot X_2(t)\\
&=
X_1(t)\cos X_1(t)\sin u(t)
+
\cos X_1(t)u(t)\\
&\quad
-
X_2(t)\sin X_2(t)\cos u(t)
-
\sin X_2(t)u(t)^2 .
\end{aligned}
\]
Since \(X_1(0)=X_2(0)=0\), we have
\[
\sin X_1(1)+\cos X_2(1)
=
1+
\int_0^1
\frac{d}{dt}\bigl(\sin X_1(t)+\cos X_2(t)\bigr)\,dt .
\]
Therefore the original cost functional differs by the constant \(1\) from
\[
\widetilde J(u(\cdot))
=
\int_0^1 c(X(t))\phi(u(t))\,dt,
\]
where
\[
c(X)
=
\bigl(
0,\,
X_1\cos X_1,\,
MX_0+\cos X_1,\,
-X_2\sin X_2,\,
-\sin X_2
\bigr).
\]
Thus the transformed problem has the structural form of \((P_2)\), up to the irrelevant
additive constant \(1\) in the value functional.

However, the convexity condition in Proposition~\ref{pj13} is still not satisfied. Indeed, take
\[
X=(0,0,0).
\]
Then
\[
Q(0)y
=
\begin{pmatrix}
y_0\\
y_2\\
y_4
\end{pmatrix},
\qquad
y=(y_0,y_1,y_2,y_3,y_4)\in \mathbb R^5.
\]
Let
\[
\bar y
=
\frac12\phi(1)+\frac12\phi(4)
\in \overline{\operatorname{co}}\phi(U).
\]
Set
\[
\xi:=Q(0)\bar y.
\]
Since
\[
Q(0)\phi(v)
=
\begin{pmatrix}
1\\
v\\
v^2
\end{pmatrix},
\]
we get
\[
\xi
=
\frac12
\begin{pmatrix}
1\\
1\\
1
\end{pmatrix}
+
\frac12
\begin{pmatrix}
1\\
4\\
16
\end{pmatrix}
=
\begin{pmatrix}
1\\[1mm]
\dfrac52\\[1mm]
\dfrac{17}{2}
\end{pmatrix}.
\]
Hence
\[
\xi\in Q(0)\overline{\operatorname{co}}\phi(U).
\]

We now claim that no point of \(\phi(U)\) belongs to the fiber
\[
\left\{
y\in \overline{\operatorname{co}}\phi(U)
\mid
Q(0)y=\xi
\right\}.
\]
Indeed, if \(Q(0)\phi(v)=\xi\) for some \(v\in U\), then
\[
\begin{pmatrix}
1\\
v\\
v^2
\end{pmatrix}
=
\begin{pmatrix}
1\\[1mm]
\dfrac52\\[1mm]
\dfrac{17}{2}
\end{pmatrix}.
\]
The second component gives \(v=5/2\), which is impossible because
\[
U=\{1,2,3,4\}.
\]
Therefore
\[
\phi(U)\cap
\left\{
y\in \overline{\operatorname{co}}\phi(U)
\mid
Q(0)y=\xi
\right\}
=
\emptyset .
\]

On the other hand, the fiber is nonempty because it contains \(\bar y\). Moreover, at \(X=0\),
\[
c(0)=(0,0,1,0,0),
\]
and therefore
\[
c(0)y=y_2.
\]
For every \(y\) in the above fiber, \(Q(0)y=\xi\) implies
\[
y_2=\frac52.
\]
Hence \(c(0)y\) is constant on the whole fiber. In particular, every feasible \(y\) in the
fiber is a minimizer of
\[
\min_{\substack{
y\in \overline{\operatorname{co}}\phi(U)\\
Q(0)y=\xi
}}
c(0)y.
\]
Since the fiber contains \(\bar y\notin\phi(U)\), the above minimum is not attained only in
\(\phi(U)\).

Consequently, even after transforming the original problem into the structural form of
Proposition~\ref{pj13}, condition~\eqref{th13} fails. Thus
Example~\ref{ex31} is not covered by Proposition~\ref{pj13}.

\end{remark}

\begin{remark}\label{rem:role_of_existence_theorem_in_example}
At first sight, one may expect the constant control $u(t)\equiv1$ to be
optimal, especially when $M$ is large, because the running cost
\begin{align*}
    \int_0^1 Mt\,u(t)\,dt
\end{align*}
penalizes larger values of $u$.  However, this is not the case.

Set
\begin{align*}
    \bar u(t)\equiv 1,
    \qquad t\in[0,1],
\end{align*}
and denote by $\bar x(\cdot)$ the corresponding state.  Then
\begin{align*}
    \dot{\bar x}_1(t)=\bar x_1(t)\sin 1+1,
    \qquad
    \dot{\bar x}_2(t)=\bar x_2(t)\cos 1+1,
    \qquad
    \bar x(0)=0.
\end{align*}
Hence
\begin{align}\label{eq:bar_x_explicit}
    \bar x_1(t)
    =
    \frac{e^{(\sin 1)t}-1}{\sin 1},
    \qquad
    \bar x_2(t)
    =
    \frac{e^{(\cos 1)t}-1}{\cos 1},
    \qquad t\in[0,1].
\end{align}
In particular,
\begin{align}\label{eq:bar_x_terminal_explicit}
    \bar x_1(1)
    =
    \frac{e^{\sin 1}-1}{\sin 1},
    \qquad
    \bar x_2(1)
    =
    \frac{e^{\cos 1}-1}{\cos 1}.
\end{align}

Let $0<\varepsilon<1$, and define
\begin{align*}
u_\varepsilon(t)
=
\begin{cases}
2, & t\in[0,\varepsilon],\\
1, & t\in(\varepsilon,1].
\end{cases}
\end{align*}
We now compute the corresponding state $x^{u_\varepsilon}(\cdot)$ explicitly.

For $t\in[0,\varepsilon]$, we have
\begin{align*}
    \dot x_1^{u_\varepsilon}(t)
    =
    x_1^{u_\varepsilon}(t)\sin 2+2,
    \qquad
    \dot x_2^{u_\varepsilon}(t)
    =
    x_2^{u_\varepsilon}(t)\cos 2+4,
    \qquad
    x^{u_\varepsilon}(0)=0.
\end{align*}
Therefore
\begin{align}\label{eq:x_eps_first_part}
    x_1^{u_\varepsilon}(t)
    =
    \frac{2\bigl(e^{(\sin 2)t}-1\bigr)}{\sin 2},
    \qquad
    x_2^{u_\varepsilon}(t)
    =
    \frac{4\bigl(e^{(\cos 2)t}-1\bigr)}{\cos 2},
    \qquad t\in[0,\varepsilon].
\end{align}
For $t\in[\varepsilon,1]$, the control becomes $1$.  Hence
\begin{align*}
    \dot x_1^{u_\varepsilon}(t)
    =
    x_1^{u_\varepsilon}(t)\sin 1+1,
    \qquad
    \dot x_2^{u_\varepsilon}(t)
    =
    x_2^{u_\varepsilon}(t)\cos 1+1.
\end{align*}
Using the values at $t=\varepsilon$ obtained from
\eqref{eq:x_eps_first_part}, the variation-of-constants formula gives
\begin{align}\label{eq:x_eps_second_part}
    x_1^{u_\varepsilon}(t)
    &=
    e^{(\sin 1)(t-\varepsilon)}
    \frac{2\bigl(e^{(\sin 2)\varepsilon}-1\bigr)}{\sin 2}
    +
    \frac{e^{(\sin 1)(t-\varepsilon)}-1}{\sin 1},
    \\
    x_2^{u_\varepsilon}(t)
    &=
    e^{(\cos 1)(t-\varepsilon)}
    \frac{4\bigl(e^{(\cos 2)\varepsilon}-1\bigr)}{\cos 2}
    +
    \frac{e^{(\cos 1)(t-\varepsilon)}-1}{\cos 1},
\end{align}
for $t\in[\varepsilon,1]$.

We next compute the difference between $x^{u_\varepsilon}(\cdot)$ and
$\bar x(\cdot)$.  For $t\in[0,\varepsilon]$, by
\eqref{eq:bar_x_explicit} and \eqref{eq:x_eps_first_part},
\begin{align}\label{eq:x_diff_first_part}
    x_1^{u_\varepsilon}(t)-\bar x_1(t)
    &=
    \frac{2\bigl(e^{(\sin 2)t}-1\bigr)}{\sin 2}
    -
    \frac{e^{(\sin 1)t}-1}{\sin 1},
    \\
    x_2^{u_\varepsilon}(t)-\bar x_2(t)
    &=
    \frac{4\bigl(e^{(\cos 2)t}-1\bigr)}{\cos 2}
    -
    \frac{e^{(\cos 1)t}-1}{\cos 1}.
\end{align}
For $t\in[\varepsilon,1]$, by \eqref{eq:bar_x_explicit} and
\eqref{eq:x_eps_second_part}, we have
\begin{align}\label{eq:x_diff_second_part}
    x_1^{u_\varepsilon}(t)-\bar x_1(t)
    &=
    e^{(\sin 1)t}
    \left[
    e^{-(\sin 1)\varepsilon}
    \frac{2\bigl(e^{(\sin 2)\varepsilon}-1\bigr)}{\sin 2}
    +
    \frac{e^{-(\sin 1)\varepsilon}-1}{\sin 1}
    \right],
    \\
    x_2^{u_\varepsilon}(t)-\bar x_2(t)
    &=
    e^{(\cos 1)t}
    \left[
    e^{-(\cos 1)\varepsilon}
    \frac{4\bigl(e^{(\cos 2)\varepsilon}-1\bigr)}{\cos 2}
    +
    \frac{e^{-(\cos 1)\varepsilon}-1}{\cos 1}
    \right].
\end{align}
In particular, at $t=1$,
\begin{align}\label{eq:Delta1_Delta2_def}
    x_1^{u_\varepsilon}(1)-\bar x_1(1)
    &=
    \Delta_1(\varepsilon),
    \\
    x_2^{u_\varepsilon}(1)-\bar x_2(1)
    &=
    \Delta_2(\varepsilon),
\end{align}
where
\begin{align}\label{eq:Delta_explicit}
    \Delta_1(\varepsilon)
    &=
    e^{\sin 1}
    \left[
    e^{-(\sin 1)\varepsilon}
    \frac{2\bigl(e^{(\sin 2)\varepsilon}-1\bigr)}{\sin 2}
    +
    \frac{e^{-(\sin 1)\varepsilon}-1}{\sin 1}
    \right],
    \\
    \Delta_2(\varepsilon)
    &=
    e^{\cos 1}
    \left[
    e^{-(\cos 1)\varepsilon}
    \frac{4\bigl(e^{(\cos 2)\varepsilon}-1\bigr)}{\cos 2}
    +
    \frac{e^{-(\cos 1)\varepsilon}-1}{\cos 1}
    \right].
\end{align}
By Taylor's expansion at $\varepsilon=0$, we obtain
\begin{align}\label{eq:Delta_asymptotic}
    \Delta_1(\varepsilon)
    =
    e^{\sin 1}\varepsilon+O(\varepsilon^2),
    \qquad
    \Delta_2(\varepsilon)
    =
    3e^{\cos 1}\varepsilon+O(\varepsilon^2).
\end{align}

Now we compute the difference of costs.  Since
$u_\varepsilon(t)-\bar u(t)=1$ on $[0,\varepsilon]$ and is equal to $0$
on $(\varepsilon,1]$, we have
\begin{align*}
    \int_0^1 Mt\,u_\varepsilon(t)\,dt
    -
    \int_0^1 Mt\,\bar u(t)\,dt
    =
    M\int_0^\varepsilon t\,dt
    =
    \frac{M}{2}\varepsilon^2.
\end{align*}
Therefore
\begin{align}\label{eq:J_diff_exact}
J(u_\varepsilon(\cdot))-J(\bar u(\cdot))
&=
\frac{M}{2}\varepsilon^2
+
\sin\bigl(\bar x_1(1)+\Delta_1(\varepsilon)\bigr)
-
\sin\bar x_1(1)
\notag\\
&\quad
+
\cos\bigl(\bar x_2(1)+\Delta_2(\varepsilon)\bigr)
-
\cos\bar x_2(1).
\end{align}
Using
\begin{align*}
    \sin(y+\eta)-\sin y
    =
    \eta\cos y+O(\eta^2),
    \qquad
    \cos(y+\eta)-\cos y
    =
    -\eta\sin y+O(\eta^2),
\end{align*}
together with \eqref{eq:Delta_asymptotic}, we get
\begin{align}\label{eq:J_diff_asymptotic}
J(u_\varepsilon(\cdot))-J(\bar u(\cdot))
&=
\left[
e^{\sin 1}\cos\bar x_1(1)
-
3e^{\cos 1}\sin\bar x_2(1)
\right]\varepsilon
\notag\\
&\quad
+
O(\varepsilon^2)
+
\frac{M}{2}\varepsilon^2.
\end{align}
By \eqref{eq:bar_x_terminal_explicit},
\begin{align*}
\bar x_1(1)
=
\frac{e^{\sin 1}-1}{\sin 1},
\qquad
\bar x_2(1)
=
\frac{e^{\cos 1}-1}{\cos 1}.
\end{align*}
A direct numerical computation gives
\begin{align*}
e^{\sin 1}\cos\bar x_1(1)
-
3e^{\cos 1}\sin\bar x_2(1)
\approx
-4.9907<0.
\end{align*}
Hence the negative linear term in \eqref{eq:J_diff_asymptotic} dominates
the quadratic terms for all sufficiently small $\varepsilon>0$.  Therefore
\begin{align*}
    J(u_\varepsilon(\cdot))<J(\bar u(\cdot))
\end{align*}
for all sufficiently small $\varepsilon>0$.  Thus the constant control
$u(t)\equiv1$ is not optimal.

\end{remark}

\begin{remark}\label{rem:PMP_candidates_not_finite}
Although the existence of an optimal control in Example
\ref{ex31} follows from Theorem \ref{peo}, the
Pontryagin maximum principle (PMP) itself does not determine the optimal control.
The reason is that the (PMP) gives only a switching rule, and the possible
switching times are not fixed in advance.

Indeed, let $u(\cdot)\in\mathcal U$ be a control satisfying the (PMP)
necessary condition, and let $x^u(\cdot)$ and $p^u(\cdot)$ be the
corresponding state and adjoint arcs.  Then, by the estimates obtained in
Step 2 of Example \ref{ex31}, we have
\begin{align}\label{eq:remark_estimates_xp}
& |x_1^u(t)|\leq 4(e^t-1),
\qquad
|x_2^u(t)|\leq 16(e^t-1),
\\
& |p_1^u(t)|\leq e^{1-t},
\qquad
|p_2^u(t)|\leq e^{1-t},
\\
& |\dot x_1^u(t)|\leq 4e^t,
\qquad
|\dot x_2^u(t)|\leq 16e^t,
\\
& |\dot p_1^u(t)|\leq e^{1-t},
\qquad
|\dot p_2^u(t)|\leq e^{1-t},
\end{align}
for every $t\in[0,1]$ in the first two lines and for a.e. $t\in[0,1]$ in
the last two lines.

For $v,w\in U$ with $v\neq w$, define
\begin{align*}
G_{vw}^u(t)
&:=
H(t,x^u(t),p^u(t),v)
-
H(t,x^u(t),p^u(t),w)
\\
&=
A_{vw}^u(t)-Mt(v-w),
\end{align*}
where
\begin{align*}
A_{vw}^u(t)
&=
p_1^u(t)x_1^u(t)(\sin v-\sin w)
+
p_1^u(t)(v-w)
\\
&\quad
+
p_2^u(t)x_2^u(t)(\cos v-\cos w)
+
p_2^u(t)(v^2-w^2).
\end{align*}
Using \eqref{eq:remark_estimates_xp}, exactly as in Step 4 of Example
\ref{ex31}, we obtain
\begin{align}\label{eq:remark_A_derivative_bound}
|\dot A_{vw}^u(t)|
\leq
(2e-1)
\bigl[
4(\sin 1-\sin 4)
+
16(\cos 1-\cos 4)
\bigr]
+
18
\end{align}
for a.e. $t\in[0,1]$ and for all $v,w\in U$ with $v\neq w$.  Hence, by
\eqref{lbm}, if $v>w$, then
\begin{align*}
\dot G_{vw}^u(t)
=
\dot A_{vw}^u(t)-M(v-w)
<0,
\qquad {\rm a.e.}\ t\in[0,1],
\end{align*}
while if $v<w$, then
\begin{align*}
\dot G_{vw}^u(t)
=
\dot A_{vw}^u(t)+M(w-v)
>0,
\qquad {\rm a.e.}\ t\in[0,1].
\end{align*}
Therefore, for every pair $v\neq w$, the function $G_{vw}^u$ is strictly
monotone on $[0,1]$. Consequently, the equation
\begin{align*}
G_{vw}^u(t)=0
\end{align*}
has at most one solution in $[0,1]$.  In other words,
\begin{align*}
\#\{t\in[0,1]\mid G_{vw}^u(t)=0\}\leq 1,
\end{align*}
where $\#E$ denotes the cardinality of the set $E$. Set
\begin{align*}
E_u
:=
\bigcup_{\substack{v,w\in U\\ v<w}}
\{t\in[0,1]\mid G_{vw}^u(t)=0\}.
\end{align*}
Since $U=\{1,2,3,4\}$, the set $E_u$ contains at most six points.

We now explain why possible switchings of a (PMP) candidate can occur only at
points of $E_u$.  Let $I$ be a connected component of $[0,1]\setminus E_u$.
Then, for every $v,w\in U$ with $v\neq w$, one has
\begin{align*}
G_{vw}^u(t)\neq 0,
\qquad t\in I.
\end{align*}
Since $G_{vw}^u$ is continuous, its sign cannot change on $I$.  Hence the
ordering of the four Hamiltonian branches
\begin{align*}
H(t,x^u(t),p^u(t),1),\quad
H(t,x^u(t),p^u(t),2),\quad
H(t,x^u(t),p^u(t),3),\quad
H(t,x^u(t),p^u(t),4)
\end{align*}
is fixed on $I$.  Therefore the maximizer of
$v\mapsto H(t,x^u(t),p^u(t),v)$ is a fixed element of $U$ on $I$.  By the
(PMP) maximum condition, $u(t)$ must be equal to this fixed element for a.e.
$t\in I$.

Thus, after modifying $u(\cdot)$ on a set of measure zero, every (PMP)
candidate is piecewise constant, and its possible switching times are
contained in the finite set $E_u$.  In particular, each (PMP) candidate has at
most six switching times.  However, the points of $E_u$ are not prescribed
in advance; they depend on $x^u(\cdot)$ and $p^u(\cdot)$, and hence on the
whole control $u(\cdot)$.  Consequently, the (PMP) yields only a finite
switching structure for each candidate, but it does not reduce the family of
(PMP) candidates to a finite set.
This is only an  a priori bound on the number of switching times of each
candidate.  It does not determine where these switching times are.  More
precisely, the equation
\begin{align*}
G_{vw}^u(t)=0
\end{align*}
depends on $x^u(\cdot)$ and $p^u(\cdot)$, and both of them depend on the
whole control $u(\cdot)$.  Hence the possible switching times are not
preassigned numbers depending only on the data of the problem.  They are
unknown continuous parameters coupled with the forward state equation and
the backward adjoint equation.

Therefore, even though the (PMP) implies that each candidate is piecewise
constant with at most six possible switching times, it does not reduce the
family of (PMP) candidates to a finite list of admissible controls.  In other
words, the (PMP) gives only a finite-switching structure for each candidate,
but not finitely many candidates.  This is why, in Example
\ref{ex31}, the existence of an optimal control
is obtained by the compactness argument in Proposition
\ref{scspro} and Theorem \ref{peo}, rather than by
enumerating all (PMP) candidates.

\end{remark}

\subsection{Finitely many (PMP) candidates do not guarantee existence of an optimal control}
Even when the (PMP) reduces the set of candidates to finitely many controls, the existence of an optimal control is still not guaranteed.
This shows that the (PMP) alone is not sufficient for existence theory.

\begin{example}\label{evpspmp}
Let \(T=1\) and
\[
    U=\{1,2,3,4\}.
\]
Define a map \(a:U\to\mathbb{R}\) by
\[
    a(1)=2,\qquad a(2)=1,\qquad a(3)=-1,\qquad a(4)=-2.
\]
Consider the scalar control system
\begin{align}\label{exau2}
\left\{
\begin{array}{l}
    \dot x(t)=a(u(t)),\quad a.e.\ t\in(0,1),
    \\[1mm]
    x(0)=0,
    \\[1mm]
    u(\cdot)\in\mathcal U,
\end{array}
\right.
\end{align}
where \(\mathcal U\) is defined by \eqref{mcs1}. Define the cost functional by
\[
    J(u(\cdot))
    =
    \int_0^1 x(t)^2\,dt
    +
    h(x(1)),
\]
where
\[
    h(x)=x^2(2-x),\quad x\in\mathbb{R}.
\]
We show that \(u(t)\equiv 1\) is the unique control satisfying the (PMP), but the optimal control does not exist.
\end{example}

\begin{proof}
The proof consists of five steps.

{\it Step 1.}
For this problem, the Hamiltonian function is
\[
    H(x,p,u)=p\,a(u)-x^2,
    \quad (x,p,u)\in\mathbb{R}\times \mathbb{R}\times U.
\]
Assume that \(\bar u(\cdot)\in\mathcal U\) satisfies the (PMP). Denote by \(x(\cdot)\) the corresponding solution to \eqref{exau2}.
The adjoint equation and the transversality condition are
\begin{align}\label{adoe1}
\left\{
\begin{array}{l}
    \dot p(t)=2x(t),\quad a.e.\,t\in(0,1),
    \\[1mm]
    p(1)=-4x(1)+3x(1)^2.
\end{array}
\right.
\end{align}
The (PMP) gives
\[
    \bar u(t)\in
    \operatorname*{arg\,max}_{v\in U}
    \{p(t)a(v)-x(t)^2\},
    \qquad a.e.\,t\in(0,1).
\]
Since the term \(-x(t)^2\) is independent of \(v\), this is equivalent to
\[
    \bar u(t)\in
    \operatorname*{arg\,max}_{v\in U} p(t)a(v),
    \qquad a.e.\,t\in(0,1).
\]
Thus, if \(p(t)>0\), then the unique maximizer is \(v=1\), because
\(a(1)=2\) is the largest value of \(a(v)\). If \(p(t)<0\), then the unique
maximizer is \(v=4\), because \(a(4)=-2\) is the smallest value of \(a(v)\).
Hence
\begin{align}\label{upgp1}
\bar u(t)=
\left\{
\begin{array}{ll}
1, & \text{if }p(t)>0,
\\[1mm]
4, & \text{if }p(t)<0,
\end{array}
\right.
\end{align}
and the corresponding solution \(x(\cdot)\) satisfies
\begin{align}\label{dxtp}
\dot x(t)=
\left\{
\begin{array}{ll}
2, & \text{if }p(t)>0,
\\[1mm]
-2, & \text{if }p(t)<0.
\end{array}
\right.
\end{align}

{\it Step 2.} We claim that
\begin{align}\label{pt00}
    |\{t\in[0,1]\mid p(t)=0\}|=0.
\end{align}
Indeed, set
\[
    E:=\{t\in[0,1]\mid p(t)=0\}.
\]
Suppose, by contradiction, that \(|E|>0\). By applying the Lebesgue
differentiation theorem \cite[Theorem 3.21, p.98]{MR1681462} to the
indicator function of the set \(E\), almost every point \(t\in E\) is a
point of density of the set \(E\) \cite[p.177]{MR924157}, namely,
\begin{align*}
\lim_{\delta\to0^+}
\frac{|E\cap(t-\delta,t+\delta)|}{2\delta}=1.
\end{align*}
Since \(p(\cdot)\) is absolutely continuous, it is differentiable almost
everywhere. Thus, we may take \(t\in E\cap(0,1)\) such that \(t\) is a
point of density of \(E\) and \(p(\cdot)\) is differentiable at \(t\). Then
there exists a sequence
\(\{t_k\}_{k\geq1}\subseteq E\setminus\{t\}\) such that \(t_k\to t\) as
\(k\to+\infty\). Since \(p(t_k)=p(t)=0\), we have
\begin{align*}
\lim_{k\to+\infty}\frac{p(t_k)-p(t)}{t_k-t}=0.
\end{align*}
Since \(p(\cdot)\) is differentiable at \(t\), this implies that
\(\dot p(t)=0\). Since such points \(t\in E\) form a full measure subset of
\(E\), we have
\begin{align*}
\dot p(t)=0,\quad a.e.\,t\in E.
\end{align*}

On the other hand, the adjoint equation \eqref{adoe1} gives
\[
    \dot p(t)=2x(t),
    \qquad \text{for a.e. }t\in[0,1].
\]
Hence
\[
    x(t)=0,
    \qquad \text{for a.e. }t\in E.
\]
Set
\[
    E_1:=E\cap\{t\in[0,1]\mid x(t)=0\}.
\]
Then \(|E_1|=|E|>0\). By applying the Lebesgue differentiation theorem to
the indicator function of \(E_1\), almost every point \(\hat t\in E_1\) is
a point of density of \(E_1\). Since \(x(\cdot)\) is absolutely continuous,
it is differentiable almost everywhere. Thus, we may take
\(\hat t\in E_1\cap(0,1)\) such that \(\hat t\) is a point of density of
\(E_1\) and \(x(\cdot)\) is differentiable at \(\hat t\). Then there exists
a sequence \(\{\hat t_k\}_{k\geq1}\subseteq E_1\setminus\{\hat t\}\) such
that \(\hat t_k\to\hat t\) as \(k\to+\infty\). Since
\(x(\hat t_k)=x(\hat t)=0\), we have
\begin{align*}
\lim_{k\to+\infty}
\frac{x(\hat t_k)-x(\hat t)}{\hat t_k-\hat t}=0.
\end{align*}
Since \(x(\cdot)\) is differentiable at \(\hat t\), this implies that
\(\dot x(\hat t)=0\). Since such points \(\hat t\in E_1\) form a full
measure subset of \(E_1\), we have
\begin{align*}
\dot x(t)=0,\quad a.e.\,t\in E_1.
\end{align*}

However, the state equation gives
\[
    \dot x(t)=a(\bar u(t)),
    \qquad \text{for a.e. }t\in[0,1],
\]
where
\[
    a(\bar u(t))\in\{2,1,-1,-2\},
    \qquad a.e.\,t\in(0,1).
\]
Thus
\[
    |\dot x(t)|\geq1,
    \qquad \text{for a.e. }t\in[0,1],
\]
which contradicts \(\dot x(t)=0\) for a.e. \(t\in E_1\), since
\(|E_1|>0\). Therefore
\[
    |\{t\in[0,1]\mid p(t)=0\}|=0.
\]

{\it Step 3.}\; We claim that \(p\) cannot change sign.

Indeed, it follows from \eqref{dxtp} and \eqref{pt00} that
\begin{align}\label{dxdp1}
\left\{
\begin{array}{l}
\dot x(t)=2\operatorname{sgn}p(t),\quad a.e.\,t\in(0,1),
\\[1mm]
\dot p(t)=2x(t),\quad a.e.\,t\in(0,1).
\end{array}
\right.
\end{align}
Since \(p(\cdot)\) is absolutely continuous and
\(|\{t\in[0,1]\mid p(t)=0\}|=0\), we have
\[
    \frac{d}{dt}|p(t)|=\dot p(t)\operatorname{sgn}p(t),
    \qquad a.e.\,t\in(0,1).
\]
Thus,
\begin{align*}
\frac{d}{dt}\left(x(t)^2-2|p(t)|\right)
&=2x(t)\dot x(t)-2\dot p(t)\operatorname{sgn}p(t)
\\
&=4x(t)\operatorname{sgn}p(t)-4x(t)\operatorname{sgn}p(t)
\\
&=0,
\quad a.e.\,t\in(0,1).
\end{align*}
Since \(x(\cdot)^2-2|p(\cdot)|\) is absolutely continuous, we have
\begin{align*}
x(t)^2-2|p(t)|
&=x(0)^2-2|p(0)|
=-2|p(0)|,
\quad \forall\,t\in[0,1].
\end{align*}
Equivalently,
\begin{align}\label{pxidentity}
    |p(t)|=\frac12 x(t)^2+|p(0)|,
    \quad \forall\,t\in[0,1].
\end{align}

If \(p(0)\neq0\), then \eqref{pxidentity} gives
\[
    |p(t)|\geq |p(0)|>0,
    \qquad \forall\,t\in[0,1].
\]
Together with the continuity of \(p(\cdot)\), this implies that \(p(\cdot)\)
cannot change sign on \((0,1)\).

It remains to consider the case \(p(0)=0\). In this case, \eqref{pxidentity}
becomes
\[
    |p(t)|=\frac12 x(t)^2,
    \qquad t\in[0,1].
\]
In particular, if \(p(\tau)=0\), then \(x(\tau)=0\). Moreover, since
\(x(\cdot)\) is continuous and \(\dot p(t)=2x(t)\) for a.e. \(t\in(0,1)\),
we have \(p\in C^1([0,1])\) and
\[
    \dot p(\tau)=2x(\tau),
    \qquad \tau\in[0,1].
\]

By contradiction, suppose that \(p(\cdot)\) changes sign on \((0,1)\). Since
\(p(0)=0\), there exists a connected component \(I=(\alpha,\beta)\) of either
\[
    \{t\in(0,1)\mid p(t)>0\}
    \quad\text{or}\quad
    \{t\in(0,1)\mid p(t)<0\}
\]
such that
\[
    0\leq \alpha<\beta<1,
    \qquad p(\alpha)=p(\beta)=0,
\]
where, if \(\alpha=0\), the equality \(p(\alpha)=0\) means \(p(0)=0\).

If \(p(t)>0\) on \(I\), then by \eqref{dxdp1},
\[
    \ddot p(t)=4,
    \qquad a.e.\,t\in I.
\]
Hence \(p(\cdot)\) is strictly convex on \([\alpha,\beta]\). Since
\(p(\alpha)=p(\beta)=0\), the definition of strict convexity gives
\[
    p(t)<0,
    \qquad \forall\,t\in(\alpha,\beta),
\]
which contradicts \(p(t)>0\) on \(I\).

If \(p(t)<0\) on \(I\), then by \eqref{dxdp1},
\[
    \ddot p(t)=-4,
    \qquad a.e.\,t\in I.
\]
Hence \(-p(\cdot)\) is strictly convex on \([\alpha,\beta]\). Since
\(-p(\alpha)=-p(\beta)=0\), the definition of strict convexity gives
\[
    -p(t)<0,
    \qquad \forall\,t\in(\alpha,\beta),
\]
which contradicts \(p(t)<0\) on \(I\).

Therefore \(p(\cdot)\) cannot change sign on \((0,1)\). Combining this with
\eqref{pt00}, we conclude that either
\[
    p(t)>0,
    \qquad a.e.\,t\in(0,1),
\]
or
\[
    p(t)<0,
    \qquad a.e.\,t\in(0,1).
\]

{\it Step 4.}\; We claim that \(\bar u(t)\equiv 1\) for a.e. \(t\in[0,1]\).

We consider two cases. If \(p(t)>0\) a.e. on \((0,1)\), then \eqref{upgp1}
gives
\[
    \bar u(t)\equiv1,
    \quad a.e.\,t\in(0,1).
\]
Hence
\[
    x(t)=2t,
    \quad t\in[0,1].
\]
Then the adjoint equation \eqref{adoe1} gives
\[
\left\{
\begin{array}{l}
\dot p(t)=4t,
\quad a.e.\,t\in(0,1),
\\[1mm]
p(1)=4,
\end{array}
\right.
\]
and therefore
\[
    p(t)=2+2t^2>0,
    \qquad t\in[0,1].
\]
This gives one (PMP) candidate, namely
\[
    \bar u(t)\equiv1,
    \quad a.e.\,t\in(0,1).
\]

If \(p(t)<0\) a.e. on \((0,1)\), then \eqref{upgp1} gives
\[
    \bar u(t)\equiv4,
    \quad a.e.\,t\in(0,1).
\]
Hence
\[
    x(t)=-2t,
    \quad t\in[0,1].
\]
The adjoint equation \eqref{adoe1} gives
\[
\left\{
\begin{array}{l}
    \dot p(t)=-4t,
    \quad a.e.\,t\in(0,1),
    \\[1mm]
    p(1)=20,
\end{array}
\right.
\]
and therefore
\[
    p(t)=22-2t^2>0,
    \quad t\in[0,1],
\]
which contradicts the assumption \(p(t)<0\) a.e. on \((0,1)\). Hence no
(PMP) candidate with \(p<0\) exists. Thus, we have
proved the claim.

{\it Step 5.}\;
We now construct admissible controls with smaller cost.

Direct computations give
\[
    J(\bar u(\cdot))
    =
    \int_0^1 4t^2\,dt
    +
    h(2)
    =
    \frac43.
\]

For each \(N\in\mathbb{N}^+\), set
\[
u_N(t)=
\begin{cases}
2, &
t\in\left[\dfrac{j}{N},\dfrac{j}{N}+\dfrac{1}{2N}\right),
\\[2mm]
3, &
t\in\left[\dfrac{j}{N}+\dfrac{1}{2N},\dfrac{j+1}{N}\right),
\end{cases}
\]
for \(j=0,1,\ldots,N-1\), and define, for instance, \(u_N(1)=2\).
Then \(u_N(\cdot)\in\mathcal U\). Since
\[
    a(2)=1,
    \qquad a(3)=-1,
\]
the corresponding state \(x_N(\cdot)\) is given by
\[
x_N(t)=
\begin{cases}
t-\dfrac{j}{N},
&
t\in
\left[\dfrac{j}{N},\dfrac{j}{N}+\dfrac{1}{2N}\right),
\\[3mm]
\dfrac{j+1}{N}-t,
&
t\in
\left[\dfrac{j}{N}+\dfrac{1}{2N},\dfrac{j+1}{N}\right),
\end{cases}
\]
for \(j=0,1,\ldots,N-1\), and \(x_N(1)=0\). In particular,
\[
    h(x_N(1))=h(0)=0.
\]
Moreover,
\[
\begin{aligned}
\int_{j/N}^{(j+1)/N} x_N(t)^2\,dt
&=
2\int_0^{1/(2N)} s^2\,ds
=
\frac{1}{12N^3}.
\end{aligned}
\]
Summing over \(j=0,1,\ldots,N-1\), we get
\[
    \int_0^1 x_N(t)^2\,dt
    =
    \frac{1}{12N^2}.
\]
Consequently,
\[
    J(u_N(\cdot))
    =
    \frac{1}{12N^2}.
\]
In particular,
\[
    J(u_N(\cdot))<\frac43=J(\bar u(\cdot))
\]
for every \(N\in\mathbb{N}^+\). Hence \(\bar u(\cdot)\) is not an optimal
control.

If an optimal control existed, then it would satisfy the (PMP). By Step 4, it would have to coincide with \(\bar u(\cdot)\).
However, \(\bar u(\cdot)\) is not optimal. This contradiction shows that
the optimal control does not exist. The proof is complete.
\end{proof}

\subsection{Non-necessity of the (PS)-type condition for existence}

We finally show that the \((PS)\)-type condition is not necessary for the
existence of optimal controls.
\begin{example}\label{edsps}
Let $T=1$,  $n\in\mathbb N^+$ be fixed, let
$
    U=\{1,2,3,4\}$,
and let  $\mathcal U$ be defined by (\ref{mcs1}).
We  equip
\(\mathcal U\) with the Ekeland metric
\[
    d(u,v)
    =
    \left|\{t\in[0,1]\mid u(t)\neq v(t)\}\right|.
\]
Consider the trivial control system
\[
\left\{
\begin{array}{l}
 \dot x(t)=0,\quad a.e.\ t\in(0,1),
 \\[1mm]
 x(0)=0,
\end{array}
\right.
\]
where \(x(t)\in\mathbb R^n\), and the cost functional
\[
    J(u(\cdot))\equiv0,
    \qquad u(\cdot)\in\mathcal U.
\]
Then every admissible control is optimal. We show that this problem does
not satisfy the (PS)-type condition.
\end{example}

\begin{proof}
For \(k\in\mathbb N^+\), set
\[
    A_k
    =
    \bigcup_{j=0}^{2^{k-1}-1}
    \left[
        \frac{2j+1}{2^k},
        \frac{2j+2}{2^k}
    \right).
\]
Thus \(A_k\) is obtained by dividing \([0,1]\) into \(2^k\) equal
subintervals and taking every other one. For instance,
\[
    A_1=\left[\frac12,1\right),
\]
\[
    A_2=
    \left[\frac14,\frac12\right)
    \cup
    \left[\frac34,1\right),
\]
and
\[
    A_3=
    \left[\frac18,\frac14\right)
    \cup
    \left[\frac38,\frac12\right)
    \cup
    \left[\frac58,\frac34\right)
    \cup
    \left[\frac78,1\right).
\]
Clearly,
\[
    |A_k|
    =
    2^{k-1}\cdot \frac1{2^k}
    =
    \frac12.
\]

Moreover, if \(k\neq l\), then
\[
    |A_k\triangle A_l|
    :=
    |(A_k\setminus A_l)\cup(A_l\setminus A_k)|
    =
    \frac12.
\]
Indeed, without loss of generality, assume that \(k<l\). For
\(j=0,1,\ldots,2^k-1\), set
\[
    I_j^{(k)}
    :=
    \left[\frac{j}{2^k},\frac{j+1}{2^k}\right).
\]
Then
\[
    [0,1)=\bigcup_{j=0}^{2^k-1} I_j^{(k)}
\]
is a disjoint union. By the definition of \(A_k\), we have
\[
    A_k\cap I_j^{(k)}
    =
    \begin{cases}
    \emptyset, & j \ \text{is even},\\[1mm]
    I_j^{(k)}, & j \ \text{is odd}.
    \end{cases}
\]
Since \(l>k\), each interval \(I_j^{(k)}\) is further divided into
\(2^{l-k}\) dyadic intervals of length \(2^{-l}\). More precisely,
\[
    I_j^{(k)}
    =
    \bigcup_{r=0}^{2^{l-k}-1}
    \left[
        \frac{j2^{l-k}+r}{2^l},
        \frac{j2^{l-k}+r+1}{2^l}
    \right).
\]
By the definition of \(A_l\), among these \(2^{l-k}\) subintervals,
exactly one half are contained in \(A_l\), and the remaining one half are
contained in \([0,1)\setminus A_l\). Hence
\[
    |A_l\cap I_j^{(k)}|
    =
    \frac12 |I_j^{(k)}|
    =
    \frac{1}{2^{k+1}}.
\]

We now compute the symmetric difference on each \(I_j^{(k)}\). If \(j\) is
even, then \(A_k\cap I_j^{(k)}=\emptyset\). Therefore
\[
\begin{aligned}
(A_k\triangle A_l)\cap I_j^{(k)}
&=
\left((A_k\setminus A_l)\cup(A_l\setminus A_k)\right)
\cap I_j^{(k)}
\\
&=
(A_l\setminus A_k)\cap I_j^{(k)}
\\
&=
A_l\cap I_j^{(k)}.
\end{aligned}
\]
Thus
\[
    |(A_k\triangle A_l)\cap I_j^{(k)}|
    =
    |A_l\cap I_j^{(k)}|
    =
    \frac12 |I_j^{(k)}|.
\]

If \(j\) is odd, then \(A_k\cap I_j^{(k)}=I_j^{(k)}\). Therefore
\[
\begin{aligned}
(A_k\triangle A_l)\cap I_j^{(k)}
&=
\left((A_k\setminus A_l)\cup(A_l\setminus A_k)\right)
\cap I_j^{(k)}
\\
&=
(A_k\setminus A_l)\cap I_j^{(k)}
\\
&=
I_j^{(k)}\setminus A_l.
\end{aligned}
\]
Thus
\[
    |(A_k\triangle A_l)\cap I_j^{(k)}|
    =
    |I_j^{(k)}|-|A_l\cap I_j^{(k)}|
    =
    \frac12 |I_j^{(k)}|.
\]
Consequently, for every \(j=0,1,\ldots,2^k-1\),
\[
    |(A_k\triangle A_l)\cap I_j^{(k)}|
    =
    \frac12 |I_j^{(k)}|
    =
    \frac{1}{2^{k+1}}.
\]
Summing over \(j=0,1,\ldots,2^k-1\), we get
\[
\begin{aligned}
    |A_k\triangle A_l|
    &=
    \sum_{j=0}^{2^k-1}
    |(A_k\triangle A_l)\cap I_j^{(k)}|
    \\
    &=
    \sum_{j=0}^{2^k-1}
    \frac{1}{2^{k+1}}
    =
    2^k\cdot \frac{1}{2^{k+1}}
    =
    \frac12.
\end{aligned}
\]

Now define
\[
u_k(t)=
\begin{cases}
1, & t\in A_k,\\
2, & t\in[0,1]\setminus A_k.
\end{cases}
\]
Then \(u_k(\cdot)\in\mathcal U\), and
\[
    J(u_k(\cdot))=0=\inf_{u(\cdot)\in \mathcal U}J(u(\cdot)).
\]
Moreover, in this trivial example we have
\[
    f(t,x,u)\equiv0,\qquad f^0(t,x,u)\equiv0,\qquad h(x)\equiv0,
\]
 and then the
Hamiltonian function is identically zero.
Therefore the approximate maximum condition in the (PS)-type condition,
see \eqref{ps3}, is automatically satisfied, for instance with \(c_k=0\).

On the other hand, for \(k\neq l\),
\[
    d(u_k,u_l)
    =|\{t\mid t\in (A_k\setminus A_l)\cup(A_l\setminus A_k)\}|
  =  |A_k\triangle A_l|
    =
    \frac12.
\]
Thus \(\{u_k(\cdot)\}\) has no Cauchy subsequence, and hence no convergent
subsequence in \((\mathcal U,d)\). Therefore the (PS)-type condition fails,
although optimal controls exist.
\end{proof}

\setcounter{equation}{0}

\section{Proofs of Theorem \ref{peo} and Proposition \ref{scspro}}\label{pmr}
\def\theequation{4.\arabic{equation}}
\hskip\parindent

\begin{proof}[Proof of Theorem \ref{peo}]

The proof is composed of four steps.

{\it Step 1.}\;
Set $C= \inf_{u(\cdot)\in\mathcal U}J(u(\cdot))$. Then,   for each $k\in\mathbb N^+$, there exists $\tilde u_k(\cdot)\in \mathcal U$ such that
\begin{align}\label{etuk}
 C\leq J(\tilde u_k(\cdot))\leq C+ \frac{1}{k}.
\end{align}
By \cite[Proposition 3.10, p.145]{ly}, the metric space $(\mathcal U,d)$ is complete. Moreover, by \cite[Lemma 4.1, p.151]{ly} and assumption (A), the map $u(\cdot)\mapsto J(u(\cdot))$ is continuous with respect to the metric $d$.
Hence, by Ekeland's variational principle, there exists $u_k(\cdot)\in\mathcal U$ such that
\begin{align}
&C\leq J(u_k(\cdot))\leq J(\tilde u_k(\cdot))\leq C+ \frac{1}{k}, \label{CJtildeJ}
\\\label{evp}& J(u(\cdot))+\frac{1}{k} d(u(\cdot), u_k(\cdot))\geq J(u_k(\cdot)),\quad \forall u(\cdot)\in\mathcal U.
\end{align}
Denote by $x_k(\cdot)$ the solution to system (\ref{cs}) corresponding to the control $u_k(\cdot)$.
According to Lemma \ref{bxec}, there exists a positive constant $\hat C$ such that
\begin{align}\label{bxk1}
|x_k(t)|\leq \hat C(1+|x_0|),\quad\forall\,t\in[0,T],\,\forall\,k\in\mathbb N^+.
\end{align}

{\it Step 2.}
Fix $k\in\mathbb N^+$.
Take any $u(\cdot)\in\mathcal U$.  From  \cite[Corollary 3.8, p.144]{ly}, we obtain that, for any $\rho\in(0,1)$, there exists a measurable subset $E_k^\rho\subset[0,T]$ such that
\begin{align}\label{lwd1}
&|E_k^\rho|=\rho T,
\\\label{lwd2}& \rho\int_0^t\left(\begin{array}{c}
f^0(s,x_k(s),u(s))-f^0(s,x_k(s),u_k(s))
\\  f(s,x_k(s),u(s))-f(s,x_k(s),u_k(s))
\end{array}
\right)ds\nonumber
\\&=\int_{(0,t)\cap E_k^\rho}\left(\begin{array}{c}
f^0(s,x_k(s),u(s))-f^0(s,x_k(s),u_k(s))
\\  f(s,x_k(s),u(s))-f(s,x_k(s),u_k(s))
\end{array}
\right)ds\nonumber
\\&\quad +\left(
\begin{array}{c}
r_{k,\rho}^0(t)
\\ r_{k,\rho}(t)
\end{array}
\right),\;\forall\,t\in[0,T],
\end{align}
where $|r_{k,\rho}^0(t)|+|r_{k,\rho}(t)|\leq \rho^2$ holds for all $t\in[0,T]$.
 Set
\begin{align}
\label{defukrho}
u_k^\rho(t)=u_k(t) I_{(E_k^\rho)^c}(t)+u(t)I_{E_k^\rho}(t),\quad\forall\,t\in[0,T],
\end{align}
where $I_S(\cdot)$ and $S^c$ denote the indicator function and the complement of a subset $S$, respectively. Let $x_k^\rho(\cdot)$ denote the solution to   system (\ref{cs}) corresponding to the control $u_k^\rho(\cdot)$.

We claim that
\begin{align}\label{fak}\begin{array}{l}
\sup_{t\in[0,T]}|x_k^\rho(t)-x_k(t)-\rho X_k(t)|=o(\rho),
\\[2mm]\left| \int_0^T f^0\left( t, x_k^\rho(t), u_k^\rho(t)\right)dt-\int_0^T f^0(t, x_k(t),u_k(t))d t-\rho Y_k^0\right|=o(\rho),
\end{array}
\end{align}
where $o(\rho)$ is a higher-order infinitesimal with respect to $\rho$,  $X_k(\cdot)$ is the solution to
\begin{align}\label{vek}
\left\{\begin{array}{l}
\dot X_k(t)=\nabla_x f(t, x_k(t),u_k(t)) X_k(t)+f(t, x_k(t),u(t))
\\ \quad\quad\;\;\;
\quad-f(t, x_k(t),u_k(t)), \quad\quad\quad\quad\quad\;a.e.\,t\in(0,T),
\\ X_k(0)=0,
\end{array}
\right.
\end{align}
and
\begin{align}
\label{defYk0}
Y_k^0=\int_0^T\left[ \nabla_x f^0(t,x_k(t),u_k(t))X_k(t)+f^0(t,x_k(t),u(t))-f^0(t,x_k(t),u_k(t))
\right]d t.
\end{align}

In fact,  the first relation of (\ref{fak}) follows directly from  \cite[Lemma 4.3, p.152]{ly}.
Now, we show that the second relation of (\ref{fak}) holds.
 By   integration by parts, the first relation of  (\ref{fak}), (\ref{bnfz1}), (\ref{lwd1}), (\ref{lwd2}) and Lemma \ref{bxec},
 we can derive
\begin{align*}
&\int_{0}^{T} \left[ f^{0}\left(t, x_{k}^{\rho}(t), u_{k}^{\rho}(t)\right) - f^{0}\left(t, x_{k} (t), u_{k}(t)\right) \right] d t
\\ &= \int_{0}^{T} \left[ f^{0}\left(t, x_{k}^\rho(t), u_{k}^\rho(t)\right) - f^{0}\left(t, x_{k}^\rho(t), u_{k}(t)\right) \right] d t
 \\
&\quad + \int_{0}^{T} \left[ f^{0}\left(t, x_{k}^{\rho}(t), u_{k}   (t)\right) - f^{0}\left(t, x_{k}   (t), u_{k}   (t)\right) \right] d t
\\
&= \int_{E_k^\rho   } \left[ f^{0}\left(t, x_{k}^{\rho}(t), u   (t)\right) - f^{0}\left(t, x_{k}^{\rho}(t), u_{k}   (t)\right) \right] dt
\\
&\quad + \int_{0}^{T} \int_{0}^{1} \nabla_{x} f^{0}\left(t, x_{k}   (t) + \theta (x_{k}^{\rho}(t) - x_{k}   (t)), u_{k}   (t)\right)  (x_{k}^{\rho}(t) - x_{k}   (t)) \, d\theta \, d t \\
&= \int_{E_k^\rho   } \left[ f^{0}\left(t, x_{k}(t), u (t)\right) - f^{0}\left(t, x_{k}(t), u_k   (t)\right) \right] dt
\\
&\quad + \int_{0}^{T} \int_0^1\nabla_{x} f^{0}\left(t, x_{k}   (t)+\theta(x_k^\rho(t)-x_k(t)), u_{k}   (t)\right)  (x_{k}^{\rho}(t) - x_{k}   (t)-\rho X_k(t)) \,d\theta  dt
\\&\quad +\rho\int_0^T\int_0^1\nabla_{x} f^{0}\left(t, x_{k}   (t)+\theta(x_k^\rho(t)-x_k(t)), u_{k}   (t)\right)  X_k(t)d\theta dt
\\&\quad+\int_{E_k^\rho}\left( f^0(t,x_k^\rho(t),u(t))- f^0(t,x_k(t),u(t))\right)  d t
\\&\quad-\int_{E_k^\rho}\left( f^0(t,x_k^\rho(t),u_k(t))- f^0(t,x_k(t),u_k(t))\right)  d t
\\&=\rho Y_k^0+R^\rho,
\end{align*}
where
\begin{align*}
&R^\rho
\\&=\rho\int_0^T\int_0^1\left[
\nabla_x  f^0\left(t, x_k(t)+\theta (x_k^\rho(t)-x_k(t)),u_k(t) \right)-\nabla_x f^0\left(t, x_k(t),u_k(t) \right)
\right]
\\&\quad \cdot X_k (t) d\theta d t-r_{k,\rho}^0(T)+o(\rho)
\\&\quad+\int_{E_k^\rho}\int_0^1\nabla_x f^0\left(t, x_k(t)+\theta (x_k^\rho(t)-x_k(t)),u(t) \right) (x_k^\rho(t)-x_k(t))d\theta d t
\\&\quad-\int_{E_k^\rho}\int_0^1\nabla_x f^0\left(t, x_k(t)+\theta (x_k^\rho(t)-x_k(t)),u_k(t) \right) (x_k^\rho(t)-x_k(t))d\theta d t,
\end{align*}
and $r_{k,\rho}^0(T)$ is given by  (\ref{lwd2}).

For any $t\in[0,T]$, integrating both sides of equation \eqref{vek} over the interval $[0,t]$ and applying \eqref{bnf1}, Lemma \ref{bxec}, and Gronwall's inequality, we obtain that  there exists a positive  constant $C_T$ such that
\begin{align}\label{bXK}
|X_k(t)
|\leq C_T,\quad\forall\,t\in[0,T].
\end{align}

We obtain from Lemma \ref{bxec},  (\ref{bnfz1}), (\ref{bnfz2}), (\ref{bXK})  and the first relation of  (\ref{fak}) that $R^\rho=o(\rho)$,
which implies  the second relation of (\ref{fak}).

{\it Step 3.}
 It follows  from (\ref{evp}),  (\ref{fak}),  (\ref{bxk1}), and the fact that $h$ is $C^1$,  that
\begin{align}
&-\frac{d(u_k^\rho(\cdot), u_k(\cdot))}{k}\nonumber
\\&\leq J(u_k^\rho(\cdot))-J(u_k(\cdot))\nonumber
\\&=\rho Y_k^0+\int_0^1\frac{\partial}{\partial\theta}h\left(x_k(T)+\theta(x_k^\rho(T)-x_k(T))  \right)d\theta+o(\rho)\nonumber
\\&=\rho Y_k^0+\int_0^1\nabla h\left(x_k(T)+\theta(x_k^\rho(T)-x_k(T))   \right)(x_k^\rho(T)-x_k(T))d\theta+o(\rho)\nonumber
\\&=\rho Y_k^0+\nabla h(x_k(T))(x_k^\rho(T)-x_k(T))+o(\rho)\nonumber
\\&\quad+\int_0^1\left[  \nabla h\left( x_k(T)+\theta(x_k^\rho(T)-x_k(T))\right)-\nabla h(x_k(T)) \nonumber  \right](x_k^\rho(T)-x_k(T))d\theta\nonumber
\\&=\rho Y_k^0+\rho\nabla h(x_k(T)) X_k(T)+o(\rho).
\label{estdk}
\end{align}
By means of integration by parts over $[0,T]$, we can derive from (\ref{dek}), (\ref{lwd1}), \eqref{defukrho}, (\ref{vek}), \eqref{defYk0} and \eqref{estdk} that
\begin{align*}
-\frac{T}{k}&\leq Y_k^0+\nabla h(x_k(T))X_k(T)+o(1)
\\&=\int_0^T\left[\left( \dot p_k(t)+p_k(t)\nabla_x f(t,x_k(t),u_k(t)) \right) X_k(t)\right.
\\&\quad\left.+f^0(t,x_k(t),u(t))-f^0(t,x_k(t),u_k(t))\right]dt-p_k(T) X_k(T)+o(1)
\\&= p_k(T)X_k(T)+\int_0^T\Big[- p_k(t)\dot X_k(t)+p_k(t)\nabla_x f(t,x_k(t),u_k(t)) X_k(t)
\\&\quad +f^0(t,x_k(t),u(t))-f^0(t,x_k(t),u_k(t))\Big]d t-p_k(T) X_k(T)+o(1)
\\&=\int_0^T\Big[H(t, x_k(t),p_k(t),u_k(t))-H(t, x_k(t),p_k(t),u(t))\Big]d t+o(1).
\end{align*}
As $\rho$ approaches zero, we obtain
\begin{align}\label{psp1}
\int_0^T\Big[H(t, x_k(t),p_k(t),u_k(t))-H(t, x_k(t),p_k(t),u(t))\Big]d t\geq -T/k.
\end{align}
Note that $u(\cdot)\in\mathcal U$ is arbitrary, thus inequality (\ref{psp1}) holds for all $u(\cdot)\in \mathcal U$. Since $k\in \mathbb N^+$ is also arbitrary,  according to the   (PS)-type condition, there exists a subsequence of $\{u_k(\cdot)\}$ ( still denoted by $\{u_k(\cdot)\}$ for simplicity) converges in the metric space $(\mathcal U, d)$.
According to \cite[Proposition 3.10, p.145]{ly}, the metric space  $\{\mathcal U,d\}$ is complete. Hence,  there exists $\bar u(\cdot)\in\mathcal U$ such that $\lim_{k\to+\infty}d(u_k(\cdot), \bar u(\cdot))=0$.

{\it Step 4.}\;We claim that $(\bar u(\cdot), \bar x(\cdot))$ is an optimal pair, where $\bar x(\cdot)$ is the solution to (\ref{cs}) corresponding to the control $\bar u(\cdot)$.

 In fact, it follows from  \cite[Lemma 4.1, p.151]{ly}  that there exists a positive constant $\tilde C$ such that
 \begin{align}\label{cxk}
|x_k(t)-\bar x(t)|\leq\tilde Cd(u_k(\cdot),\bar u(\cdot)),\quad \forall\,t\in[0,T], \;k\in\mathbb N^+.
\end{align}
 Then, by means of the fundamental theorem of calculus, we can obtain
\begin{align*}
 &J(u_k(\cdot))-J(\bar u(\cdot))
 \\
& = \int_{0}^{T} \left[ f^{0}\left(t, x_{k}(t), u_{k}(t)\right) - f^{0}\left(t, \bar x(t),u_k(t)\right)\right.
\\
 & \quad \left. +\, f^{0}\left(t, \bar x(t), u_k(t)\right) - f^{0}\left(t,\bar x(t), \bar u(t)\right) \right] d t+h(x_k(T))-h(\bar x(T))
 \\&=\int_0^T\int_0^1\nabla_x f^0\left(t, \bar x(t)+\theta (x_k(t)-\bar x(t)), u_k(t)  \right) (x_k(t)-\bar x(t))d\theta dt
 \\&\quad +\int_0^T\left[f^0(t,0,u_k(t)) -f^0(t,0,\bar u(t))\right]d t
 \\&\quad+\int_0^T\left[f^0(t,\bar x(t),u_k(t)) -f^0(t,0,u_k(t))\right]d t
 \\&\quad-\int_0^T\left[f^0(t,\bar x(t) ,\bar u(t))-f^0(t,0,\bar u(t))\right]d t+h(x_k(T))-h(\bar x(T))
 \\&= \int_0^T\int_0^1\nabla_x f^0\left(t, \bar x(t)+\theta (x_k(t)-\bar x(t)), u_k(t)  \right) (x_k(t)-\bar x(t))d\theta d t
 \\&\quad+\int_{\{\tau\mid u_k(\tau)\not=\bar u(\tau)\}}\left[f^0(t,0,u_k(t)) -f^0(t,0,\bar u(t))\right]d t+h(x_k(T))-h(\bar x(T))
 \\&\quad+\int_{\{\tau\mid u_k(\tau)\not=\bar u(\tau)\}}\int_0^1\left[\nabla_x f^0(t,\theta \bar x(t), u_k(t))-\nabla_x f^0(t,\theta \bar x(t), \bar u(t))\right]\bar x(t) d\theta dt.
\end{align*}
Since (\ref{cxk}),   (\ref{bnfz1})  and (\ref{bfz7}) hold and  function    $h$ is assumed to be $C^1$, we obtain from Lemma \ref{bxec} and  \eqref{CJtildeJ}  that $J(\bar u(\cdot))=C$. Thus, $(\bar u(\cdot), \bar x(\cdot))$ is an optimal pair.
 The proof is complete.

\end{proof}

\begin{proof}[Proof of Lemma \ref{bxec}]

First, we prove the estimate for \(x_u(\cdot)\). Take any
\(u(\cdot)\in\mathcal U\). By integrating equation \eqref{cs} over
\([0,t]\), we obtain from assumption (A) that
\begin{align*}
|x_u(t)|
&\leq
|x_0|
+
\int_0^t
|f(s,x_u(s),u(s))-f(s,0,u(s))|\,ds
+
\int_0^t |f(s,0,u(s))|\,ds
\\
&\leq
|x_0|
+
K\int_0^t |x_u(s)|\,ds
+
K_1t .
\end{align*}
By Gronwall's inequality,
\[
    |x_u(t)|
    \leq
    |x_0|e^{Kt}
    +
    K_1\frac{e^{Kt}-1}{K}
    =
    R(t),
    \qquad t\in[0,T].
\]

Second, we prove the estimate for \(p_u(\cdot)\). In \eqref{dek_u}, make the
change of variables
\[
    \tau=T-t,
    \qquad
    \tilde p_u(\tau):=p_u(T-\tau).
\]
Then \(\tilde p_u\) satisfies
\begin{align*}
\left\{\begin{array}{ll}
\dot{\tilde p}_u(\tau)
&=
\tilde p_u(\tau)
\nabla_x f(T-\tau,x_u(T-\tau),u(T-\tau))
\\
&\quad
-
\nabla_x f^0(T-\tau,x_u(T-\tau),u(T-\tau)),
\quad a.e.\,\tau\in(0,T),
\\[1mm]
\tilde p_u(0)
&=
-\nabla h(x_u(T)).
\end{array}
\right.
\end{align*}
Hence, using \(|x_u(t)|\leq R(T)\), we get
\begin{align*}
|\tilde p_u(\tau)|
&\leq
\sup_{|x|\leq R(T)}|\nabla h(x)|
+
K\int_0^\tau |\tilde p_u(s)|\,ds
+
C_{R(T)}^0 \tau,
\qquad \tau\in[0,T].
\end{align*}
Applying Gronwall's inequality gives
\[
    |\tilde p_u(\tau)|
    \leq
    \sup_{|x|\leq R(T)}|\nabla h(x)| e^{K\tau}
    +
    C_{R(T)}^0\frac{e^{K\tau}-1}{K},
    \qquad  \tau\in[0,T].
\]
Returning to \(t=T-\tau\), we obtain
\[
    |p_u(t)|
    \leq
    \sup_{|x|\leq R(T)}|\nabla h(x)| e^{K(T-t)}
    +
    C_{R(T)}^0\frac{e^{K(T-t)}-1}{K}
    =
    P(t),
    \qquad t\in[0,T].
\]

Finally, we prove the relative compactness. For a.e. \(t\in[0,T]\),
\begin{align*}
|\dot x_u(t)|
&=
|f(t,x_u(t),u(t))|
\\
&\leq
|f(t,x_u(t),u(t))-f(t,0,u(t))|
+
|f(t,0,u(t))|
\\
&\leq
K|x_u(t)|+K_1
\\
&\leq
KR(t)+K_1.
\end{align*}
Similarly, from the adjoint equation,
\begin{align*}
|\dot p_u(t)|
&\leq
|p_u(t)|\,|\nabla_x f(t,x_u(t),u(t))|
+
|\nabla_x f^0(t,x_u(t),u(t))|
\\
&\leq
KP(t)+C_{R(T)}^0 .
\end{align*}
Therefore, for every \(0\leq s<t\leq T\),
\[
    |x_u(t)-x_u(s)|
    \leq
    \int_s^t \bigl(KR(\tau)+K_1\bigr)\,d\tau,
\]
and
\[
    |p_u(t)-p_u(s)|
    \leq
    \int_s^t \bigl(KP(\tau)+C_{R(T)}^0\bigr)\,d\tau.
\]
These estimates imply that both families
\[
    \{x_u(\cdot)\mid u(\cdot)\in\mathcal U\},
    \qquad
    \{p_u(\cdot)\mid u(\cdot)\in\mathcal U\}
\]
are equicontinuous. Together with the uniform bounds \eqref{bxu11}, the
Arzelà--Ascoli theorem implies the relative compactness of
\[
    \{(x_u(\cdot),p_u(\cdot))\mid u(\cdot)\in\mathcal U\}
\]
in
$
C([0,T];\mathbb R^n)\times C([0,T];(\mathbb R^n)^*)$.
The proof is complete.
\end{proof}

\begin{proof}[Proof of Proposition  \ref{scspro}]

Let \(\{u_k(\cdot)\}_{k\geq 1}\subset\mathcal U\) and
\(\{c_k\}_{k\geq 1}\subset[0,+\infty)\) satisfy (\ref{ps3}). Denote by
\(x_k(\cdot)\) the solution to (\ref{cs}) corresponding to the control
\(u_k(\cdot)\), and by \(p_k(\cdot)\) the solution to (\ref{dek}).

\emph{Step 1.}
According to Lemma \ref{bxec}, there exist a subsequence of
\(\{x_k(\cdot), p_k(\cdot)\}_{k\geq 1}\), still denoted by itself, and
\((x_*(\cdot), p_*(\cdot))\in C([0,T];\mathbb R^n)\times C([0,T];(\mathbb R^n)^*)\) such that
\begin{align}\label{xkx8}
    (x_k(\cdot),p_k(\cdot))\to (x_*(\cdot), p_*(\cdot))
    \quad\text{in }C([0,T];\mathbb R^n)\times C([0,T];(\mathbb R^n)^*).
\end{align}

We claim that
\begin{align}\label{hkv8}
    \max_{v\in U}\sup_{t\in[0,T]}
    \left|
    H(t,x_k(t),p_k(t),v)
    -
    H(t,x_*(t),p_*(t),v)
    \right|
    \to0,
    \qquad\text{as }k\to+\infty.
\end{align}

In fact, for each \(v\in U\), by the definition of the Hamiltonian, we have
\begin{align*}
&\quad
\left|
H(t,x_k(t),p_k(t),v)
-
H(t,x_*(t),p_*(t),v)
\right|
\\
&\leq
|p_k(t)-p_*(t)|\,|f(t,x_k(t),v)|
\\
&\quad
+
|p_*(t)|\,|f(t,x_k(t),v)-f(t,x_*(t),v)|
\\
&\quad
+
|f^0(t,x_k(t),v)-f^0(t,x_*(t),v)|.
\end{align*}
By (\ref{bnf1}) and  Lemma \ref{bxec}, the right-hand side is
bounded by
\begin{align*}
& |p_k(t)-p_*(t)|
    \left(K|x_k(t)|+K_1\right)
    +K |p_*(t)||x_k(t)-x_*(t)|
\\
&\quad
+
\int_0^1
\left|
\nabla_x f^0\left(t,x_*(t)
+\theta(x_k(t)-x_*(t)),v\right)
\right|\,d\theta\,
|x_k(t)-x_*(t)|.
\end{align*}
Therefore, by Lemma \ref{bxec}, (\ref{bnfz1}), and (\ref{xkx8}), we get
\[
    H(t,x_k(t),p_k(t),v)
    \to
    H(t,x_*(t),p_*(t),v)
\]
uniformly in \(t\in[0,T]\). Since \(U\) has finitely many elements, we derive
(\ref{hkv8}).

\emph{Step 2.}
We claim that, for any \(\varepsilon>0\), there exists
\(N_\varepsilon\in\mathbb N^+\) such that
\begin{align}\label{eukl}
    d(u_k(\cdot),u_l(\cdot))
    \leq
    \frac{c_k}{\varepsilon}
    +
    \frac{c_l}{\varepsilon}
    +
    \left|\{t\in[0,T]\mid \Delta_*(t)\leq 4\varepsilon\}\right|,
    \qquad
    \forall\, k,l\geq N_\varepsilon,
\end{align}
where $\Delta_*(\cdot)$ is defined by (\ref{deta8}).

For each \(k\in\mathbb N^+\), set
\[
    M_k(t):=
    \max\{H(t,x_k(t),p_k(t),v_i)\mid i=1,\ldots,m\},
    \qquad t\in[0,T].
\]
Since \(U\) is finite and each function
\[
    t\mapsto H(t,x_k(t),p_k(t),v_i)
\]
is measurable, \(M_k(\cdot)\) is measurable. Moreover, by choosing the first
maximizer in the finite set \(U=\{v_1,\ldots,v_m\}\), we obtain a measurable
function \(\widetilde u_k(\cdot)\in\mathcal U\) such that
\[
    M_k(t)=H(t,x_k(t),p_k(t),\widetilde u_k(t)),
    \qquad t\in[0,T].
\]
Indeed, define
\[
    A_1^k:=
    \{t\in[0,T]\mid H(t,x_k(t),p_k(t),v_1)=M_k(t)\},
\]
and, for \(i=2,\ldots,m\),
\[
    A_i^k:=
    \{t\in[0,T]\mid H(t,x_k(t),p_k(t),v_i)=M_k(t)\}
    \setminus
    \bigcup_{j=1}^{i-1}A_j^k.
\]
Then the sets \(A_i^k\) are measurable, pairwise disjoint, and cover
\([0,T]\). Define
\[
    \widetilde u_k(t):=v_i,
    \qquad t\in A_i^k.
\]
This gives the desired measurable maximizer.

Taking \(u(\cdot)=\widetilde u_k(\cdot)\) in (\ref{ps3}), we obtain
\[
    \int_0^T
    \left[
    M_k(t)-H(t,x_k(t),p_k(t),u_k(t))
    \right]\,dt
    \leq c_k .
\]
For fixed \(\varepsilon>0\), define the bad set
\[
    B_k^\varepsilon
    :=
    \left\{t\in[0,T]\mid
    M_k(t)-H(t,x_k(t),p_k(t),u_k(t))>\varepsilon
    \right\}.
\]
By Markov's inequality,
\begin{align}\label{Bkeps}
    |B_k^\varepsilon|
    \leq
    \frac{c_k}{\varepsilon}.
\end{align}

According to the assumption (H), for almost every \(t\in[0,T]\), the function
\[
    v\mapsto H(t,x_*(t),p_*(t),v),
    \qquad v\in U,
\]
has a unique maximizer. We denote it by \(m_*(t)\), which is measurable.

By (\ref{hkv8}), there exists \(N_\varepsilon\in\mathbb N^+\) such that
\[
    \left|
    H(t,x_k(t),p_k(t),v)
    -
    H(t,x_*(t),p_*(t),v)
    \right|
    <\varepsilon,
\]
for all \((t,v)\in[0,T]\times U\) and all \(k\geq N_\varepsilon\).

Take any
\[
    t\in
    (B_k^\varepsilon)^c
    \cap
    \{\tau\in[0,T]\mid \Delta_*(\tau)>4\varepsilon\},
\]
where \((B_k^\varepsilon)^c=[0,T]\setminus B_k^\varepsilon\). Then
\[
    H(t,x_*(t),p_*(t),m_*(t))
    -
    H(t,x_*(t),p_*(t),v)
    >4\varepsilon,
    \qquad
    \forall\,v\in U\setminus\{m_*(t)\}.
\]
Hence, for all \(k\geq N_\varepsilon\) and
\(v\in U\setminus\{m_*(t)\}\),
\begin{align*}
&\quad
H(t,x_k(t),p_k(t),m_*(t))
-
H(t,x_k(t),p_k(t),v)
\\
&=
H(t,x_k(t),p_k(t),m_*(t))
-
H(t,x_*(t),p_*(t),m_*(t))
\\
&\quad
+
H(t,x_*(t),p_*(t),m_*(t))
-
H(t,x_*(t),p_*(t),v)
\\
&\quad
+
H(t,x_*(t),p_*(t),v)
-
H(t,x_k(t),p_k(t),v)
\\
&>2\varepsilon.
\end{align*}
Thus, for \(k\geq N_\varepsilon\), \(m_*(t)\) is the unique maximizer of
\[
    v\mapsto H(t,x_k(t),p_k(t),v),
    \qquad v\in U,
\]
with a gap larger than \(2\varepsilon\).

If \(u_k(t)\neq m_*(t)\), then
\[
    M_k(t)-H(t,x_k(t),p_k(t),u_k(t))>2\varepsilon>\varepsilon,
\]
which contradicts \(t\notin B_k^\varepsilon\). Therefore,
\[
    u_k(t)=m_*(t)
\]
for every
\[
    t\in
    (B_k^\varepsilon)^c
    \cap
    \{\tau\in[0,T]\mid \Delta_*(\tau)>4\varepsilon\},
    \qquad k\geq N_\varepsilon .
\]

Consequently, for \(k,l\geq N_\varepsilon\),
\[
    \{t\in[0,T]\mid u_k(t)\neq u_l(t)\}
    \subset
    B_k^\varepsilon
    \cup
    B_l^\varepsilon
    \cup
    \{t\in[0,T]\mid \Delta_*(t)\leq4\varepsilon\}.
\]
Taking measures and using (\ref{Bkeps}), we obtain (\ref{eukl}).

\emph{Step 3.}
We show that \(\{u_k(\cdot)\}_{k\geq 1}\) is a Cauchy sequence in
\((\mathcal U,d)\).

To this end, take any \(\eta>0\). Since
\[
    \{t\in[0,T]\mid \Delta_*(t)\leq4\varepsilon\}
\]
decreases to
\[
    \{t\in[0,T]\mid \Delta_*(t)=0\}
\]
as \(\varepsilon\to0^+\), and since all these sets are contained in
\([0,T]\), the continuity from above of Lebesgue measure
\cite[Theorem 1.8, p.~25]{MR1681462} implies
\[
    \left|\{t\in[0,T]\mid\Delta_*(t)\leq4\varepsilon\}\right|
    \to
    \left|\{t\in[0,T]\mid\Delta_*(t)=0\}\right|,
    \qquad
    \text{as }\varepsilon\to0^+.
\]
By (\ref{D80}), we have
\[
    \lim_{\varepsilon\to0^+}
    \left|\{t\in[0,T]\mid\Delta_*(t)\leq4\varepsilon\}\right|
    =0.
\]
Therefore, there exists \(\varepsilon_\eta>0\) such that
\[
    \left|\{t\in[0,T]\mid\Delta_*(t)\leq4\varepsilon_\eta\}\right|
    <\frac{\eta}{3}.
\]
Since \(c_k\to0\), there exists
\(\widehat N_\eta>N_{\varepsilon_\eta}\) such that
\[
    \frac{c_k}{\varepsilon_\eta}<\frac{\eta}{3},
    \qquad
    \forall\, k\geq \widehat N_\eta .
\]
Thus, for all \(k,l\geq \widehat N_\eta\), by (\ref{eukl}) we have
\[
    d(u_k(\cdot),u_l(\cdot))
    \leq
    \frac{c_k}{\varepsilon_\eta}
    +
    \frac{c_l}{\varepsilon_\eta}
    +
    \left|\{t\in[0,T]\mid \Delta_*(t)\leq4\varepsilon_\eta\}\right|
    <\eta.
\]
Hence \(\{u_k(\cdot)\}_{k\geq 1}\) is a Cauchy sequence in the Ekeland
metric \(d\).

Since \((\mathcal U,d)\) is complete by \cite[Proposition 3.10, p.~145]{ly}, the above Cauchy subsequence converges in \((\mathcal U,d)\). This proves the (PS)-type  condition.
\end{proof}

\section*{Declaration of generative AI and AI-assisted technologies}

During the preparation of this work, the authors used DeepSeek, a generative AI tool, for LaTeX editing and for obtaining suggestions in the search for suitable examples. The final examples presented in the paper were selected, formulated, and verified independently by the authors. The use of this tool was limited to the auxiliary purposes described above. The authors take full responsibility for the final content of the paper.

\bibliographystyle{amsplain}
\providecommand{\bysame}{\leavevmode\hbox to3em{\hrulefill}\thinspace}
\providecommand{\MR}{\relax\ifhmode\unskip\space\fi MR }
\providecommand{\MRhref}[2]{%
  \href{http://www.ams.org/mathscinet-getitem?mr=#1}{#2}
}
\providecommand{\href}[2]{#2}

\end{document}